\documentclass{amsart} 
\usepackage{amssymb,latexsym,amsfonts,verbatim,amscd,ifthen, graphicx} 
\usepackage{color,mathtools,tikz}

\newcommand{\C}{{\mathbb C}} 
\newcommand{\R}{{\mathbb R}} 

\newcommand{\p}{{\mathbb P}} 
 
\numberwithin{equation}{section} 
 
\theoremstyle{plain} 
 
\newtheorem{theorem}[equation]{Theorem} 
\newtheorem{lemma}[equation]{Lemma} 
\newtheorem{proposition}[equation]{Proposition} 
 
\newtheorem{problem}[equation]{Problem} 
\newtheorem*{definition*}{Definition}

\theoremstyle{definition} 
 
\newtheorem{remark}[equation]{Remark} 
\newtheorem{definition}[equation]{Definition} 
 
\DeclareMathOperator{\trace}{Tr}  
\DeclareMathOperator{\cont}{ct} 
\DeclareMathOperator{\sign}{sig}

\begin{document} 
\title 
[Spectra and Coxeter groups] 
{Determinantal hypersurfaces and %GL($n$) 
representations of Coxeter groups} 
 
\author[\v{Z}. \v{C}u\v{c}kovi\v{c}]{\v{Z}. \v{C}u\v{c}kovi\v{c}}
\address{
Department of Mathematics and Statistics \\ 
University of Toledo \\ 
Toledo, OH 43606, USA}
\email{zcuckovi@math.utoledo.edu} 

\author[M. I. Stessin]{M. I. Stessin}
\address{
Department of Mathematics and Statistics \\
University at Albany, SUNY \\
Albany, NY 12222, USA
}
\email{mstessin@albany.edu}

\author[A. B. Tchernev]{A. B. Tchernev}
\address{
Department of Mathematics and Statistics \\
University at Albany, SUNY \\
Albany, NY 12222, USA
}
\email{atchernev@albany.edu}


\begin{abstract}
Given a finite generating set $T=\{g_0,\dots, g_n\}$ of a group 
$G$, and a representation $\rho$ of $G$ on a Hilbert space $V$, 
we investigate how the geometry of the set \ 
$
D(T,\rho)=\{[x_0:\dots:x_n]\in\C\p^n \mid \sum x_i\rho(g_i) 
\text{ not invertible} \}
$ 
\ reflects the properties of $\rho$. When $V$ is finite-dimensional 
this is an algebraic hypersurface in $\C\p^n$. In the special 
case $T=G$ and $\rho=$ the left regular representation of $G$, 
this hypersurface is 
defined by the \emph{group determinant}, an object 
studied extensively in the founding work of Frobenius that lead to 
the creation of representation theory. 
We focus on the classic case when $G$ is a finite Coxeter group, 
and make $T$ by adding the identity element $1_G$ 
to a Coxeter generating set for $G$. 
Under these assumptions we show 
in our first main result 
that if $\rho$ is the left regular representation, then  
$D(T,\rho)$ determines the isomorphism class of $G$.  
Our second main result is that if $G$ is not of exceptional type, 
and  $\rho$ is any finite dimensional representation, then 
$D(T,\rho)$ determines $\rho$. 
\end{abstract}

\maketitle 
% Create a new 1st level heading 
\section{%\bf Introduction and Statement of the Main Theorem
Introduction}
%\vspace{.2cm}

Given a tuple of linear operators 
%$N\times N$ complex-valued matrices 
$A_0,\dots,A_n$, on a complex $N$-dimensional vector space $V$, 
the determinant
\begin{equation*}\label{determinant}
{\mathcal P}(x_0,\dots,x_n)= \det\left[ x_0A_0+ \dots +x_nA_n \right]
\end{equation*}
is a homogeneous polynomial in $x_0,\dots,x_n$ of degree $N$. 
Zeros of this polynomial form an algebraic closed subscheme of  
%of dimension $n-2$ in the 
complex projective space $\C\p^n$ called a 
\textbf{determinantal hypersurface},  
%of the pencil $x_1A_1+...+x_nA_n$} 
and we denote it by $\sigma(A_0,\dots,A_n)$. 

Determinantal hypersurfaces are objects with a long history in algebraic 
geometry. 
%There are two 
A major direction of research related to their study 
%of determinantal manifolds. The first one is
%: given an algebraic manifold 
is to determine which hypersurfaces in $\C\p^n$ are determinantal.   
%of codimension one, can we find matrices $A_1,...,A_n$ 
%such that this manifold is the determinantal manifold 
%determined by zeros of (\ref{determinant})? 
This is a classical avenue, %in algebraic geometry, 
see e.g. \cite{CAT, D, Do, HMV, HV, KV, R, V}.

In this paper we follow another approach to 
%The second angle of researching 
determinantal hypersurfaces,  
%is
%: given that a hypersurface in $\C\p^{n-1}$ has a determinantal %representation, what does 
where one analyzes what the geometry of $\sigma(A_0,\dots, A_n)$ 
can say about the mutual relations between the operators $A_0,\dots,A_n$. 
While this kind of question is natural from the perspective of 
operator theory,  
until recently it seems to have attracted much less 
attention in geometry. 

One of the few relevant instances of work along these lines 
%we have 
is the result of Motzkin and Taussky \cite{MT} 
%which states 
that a real curve in $\C\p^2$ having a self-adjoint determinantal 
representation with one of the three operators being invertible 
(and, therefore, assumed to be the identity) satisfies the condition: 
the operators commute if and only if this determinantal curve is a 
union of projective lines (in \cite{MT} the result is stated in 
equivalent but different form).

Another --- and the most relevant to the subject of this paper ---  
instance where the geometry of a determinantal hypersurface 
has been extensively 
studied from our point of view comes from classic work of 
Frobenius that goes back to the origins of representation theory. 
Specifically, when $G=\{g_0,\dots, g_n\}$ is a finite group and 
$A_i$ represents the action of $g_i$ by left multiplication on the 
group ring $\C[G]$ (this is usually called the \emph{left regular 
representation} of $G$), the defining equation 
%$\mathcal P(x_1,\dots, x_n)$ 
of $\sigma(A_0,\dots, A_n)$ 
is called the \emph{group determinant} of $G$. The papers of Frobenius 
\cite{Fr1, Fr2, Fr3} show, in modern language, 
that the irreducible components of 
$\sigma(A_0,\dots, A_n)$ are in bijective correspondence with the 
irreducible representations of $G$, and the multiplicity of each 
component %of $\sigma(A_1,\dots, A_n)$ 
equals the dimension of the 
corresponding irreducible representation. 
%We will be remiss 
%not mention that this work of Frobenius directly led to the 
%creation of representation theory. 
For more on the group determinant and related problems we refer the  
reader to \cite{D2, D3, D4, D5, FoS, HJ1, HJ2, J} and the references 
there. In particular, 
it was shown by Formanek and Sibley~\cite{FoS} that the group 
determinant determines the isomorphism class of the group as well. 

A substantial drawback when working with group determinants is that,  
as the size of the group rises, the resulting increase in complexity 
makes the group determinant essentially inaccessible. It is natural to 
ask if it is possible, instead of all elements, to take only a 
suitable generating set for $G$, thus computing a much simpler 
determinantal hypersurface while still retaining all the salient features 
of the group determinant.  
Furthermore, it is desirable to have analogous results also for 
finitely generated groups that are not necessarily finite. To do 
this one 
would need a generalization for the notion of a determinantal hypersurface
to the setting of operators on an infinite dimensional space.
This naturally leads us  to 
the notion of \textbf{projective joint spectrum} 
%that generalizes the notion 
%of determinantal hypersurface to the case of 
for operators acting on 
a Hilbert space $V$, which was introduced by Yang~\cite{Y}. %in 2009. 
If $A_0,\dots,A_n$ are bounded linear operators on $V$, 
their joint spectrum %of $A_0,\dots,A_n$ 
is the set  
\begin{equation*}%\label{spectrum-definition}
\sigma(A_0,\dots,A_n)=
\{ 
[x_0:\dots :x_n]\in \C\p^n: \ 
x_0A_0+\dots+x_nA_n \ \text{ is not invertible} \}.
\end{equation*}
%In general this is just a subset of $\C\p^n$. However,  
When $V$ is finite dimensional, 
%the operators are given by matrices, and the definitions of 
the joint spectrum is a determinantal hypersurface. 
%and $\sigma(A_0,\dots, A_n)$ acquires 
%the additional structure of an algebraic closed subscheme of $\C\p^n$. 
Since in this paper we predominantly deal with the 
finite dimensional case, we will use the terms ``determinantal hypersurface" 
and ``joint spectrum" interchangeably.  
The relation between the geometry of the joint spectrum (or its complement, 
the \textbf{joint resolvent set}) and the properties of the tuple 
$A_0,\dots,A_n$ was investigated in 
\cite{BCY, CY, CSZ, DY1, DY2, GS, GY, HWY, ST, SYZ, Y}. 
In \cite{GY} joint spectrum was used in relation to the 
Gelfand-Naimark-Siegal representation of the infinite dihedral group.

%Now    
%More precisely, 
Let $T=\{g_0,\dots, g_n\}$ be a set of generators of 
a group $G$, and let $\rho: G\longrightarrow GL(V)$ be a 
%complex finite dimensional linear representation of $G$
homomorphism into the group of bounded invertible linear operators on 
a Hilbert space $V$. Let 
\[
D(T,\rho) = \sigma\bigl(\rho(g_0),\dots,\rho(g_n)\bigr)
\]
be the corresponding joint spectrum. If $\rho$ is the 
left regular representation of $G$ we %drop it from the notation and 
will write just $D(T)$ (when $G$ is infinite its left regular 
representation is obtained by the left action of $G$ on the 
completion of $\C[G]$ with respect to  its inner product norm, 
see Section~\ref{background}). %instead of $D(T,\rho)$. 
The results of Frobenius and of Formanek and Sibley naturally %lead 
prompt us to ask:  
\begin{enumerate} 
\item
When does $D(T)$ determine $G$? %and 

\item
When does $D(T,\rho)$ determine $\rho$? 

\item
If $\rho$ is finite dimensional and irreducible, when is $D(T,\rho)$ 
reduced and irreducible? 
\end{enumerate}

%In this paper 
We address these %two 
fundamental structural questions in the classical 
case when $G$ is a finitely generated  
Coxeter group. We consider the generating set $T$ 
obtained by adding 
%has the form 
%$T=\{1_, g_1,\dots, g_n\}$, where $1_G$ is the 
%consists of 
the identity element $1_G$ %of $G$ to %and the $g_i$s form    
to a Coxeter generating set for $G$; see Section~\ref{background} for   
definitions. Our first main result concerns  
question (1): 

\begin{theorem}\label{T:main-theorem-0}
Let $G$ be a Coxeter group with Coxeter generating set 
$\{g_1,\dots, g_n\}$, and let $T=\{1_G, g_1,\dots, g_n\}$.   
Let $G'$ be a group, and let 
$T'=\{1_{G'}, g_1',\dots,g_n'\}$ be a generating set for $G'$. 
We have:  
%where $1'$ is the identity element of $G'$. 
%and let $f: T\longrightarrow T'$ be a bijection 
%of sets such that $f(1)=1$. 
\begin{enumerate}
\item  
If $D(T)\supseteq D(T')$ as subsets of\/ 
$\C\p^n$, then there is an 
epimorphism of groups $f : G\longrightarrow G'$ such that 
$f(g_i)=g_i'$ for each $1\le i\le n$. In particular, if $G$ 
is finite then so is $G'$. 

\item
If $G$ is finite and $D(T)=D(T')$ as subschemes of\/ 
$\C\p^n$, then the homomorphism $f$ from $(1)$ is an isomorphism.  
\end{enumerate} 
\end{theorem}  

Next, we consider question (2) for finite dimensional representations. 
It is easy to see that, when $V$ is finite dimensional, $D(T,\rho)$ 
is completely determined by the multiset of isomorphism classes 
of irreducible representations 
that appear as factors in a composition series for $V$ when considered 
as a $\C[G]$-module via $\rho$. Thus whenever $\C[G]$ is not a semisimple 
ring, there 
will always exist non-equivalent representations of $G$ having the same 
determinantal hypersurface. Note however, that this kind of obstruction 
does not arise if $G$ is finite, or if one considers either only 
unitary or only irreducible representations. 
 
In our second main result we present a positive answer to question (2) 
when $G$ is a finite Coxeter group 
of regular type. The exceptional types will be treated 
in a separate paper.

\begin{theorem}\label{main-theorem} %{\bf (Main Theorem)} 
Let $G$ and $T$ be as in Theorem~\ref{T:main-theorem-0}. 
Suppose $G$ is finite %Coxeter group 
of type either $I$ (dihedral group), 
or $A$, or $B$, or $D$. 
%and let $g_1,...,g_n$ be the Coxeter generators of $G$. 
If for two finite dimensional complex linear 
representations $\rho_1$ and $\rho_2$ of $G$ we have %an equality 
\begin{equation}\label{same spectra} 
D(T,\rho_1) = D(T,\rho_2)
\end{equation}
as subschemes of\/ $\C\p^n$, then the representations $\rho_1$ and 
$\rho_2$ are equivalent. 
\end{theorem}

In our proofs we make a conscious attempt to emphasize the direct 
connection between porperties of the generating set $T$ 
and the geometry of the corresponding determinantal hypersurface. 
Thus, in the case when $G$ is dihedral, even though there are 
alternative arguments using known results from representation 
theory, we chose to rely on the tools developed    
%The main tools that we need in our proof can be found
in the recent paper \cite{ST}, where 
the relation between the appearance of algebraic 
curves of finite multiplicity in the joint spectrum 
of a tuple of operators
and decomposability of the tuple is investigated. 
In the case of unitary self-adjoint operators the 
technique developed there can be applied to finding %a 
commutation relations between the operators, and, therefore, 
be applied to %linear 
representations of Coxeter groups. 
This method is quite general, and 
proved to be useful also for other, non-Coxeter groups. The upcoming paper \cite{STC} will contain the details.

To handle the cases $A$, $B$, and $D$ in Theorem~\ref{main-theorem}, 
we demonstrate that in those types  
$D(T,\rho)$ encodes the character of the representation $\rho$. 
We accomplish this by leveraging an explicit combinatorial algorithm 
that works uniformly in all three regular types, and 
transforms an element $g$ of $G$ into what we call  
{\bf echelon form} -- an element with good combinatorial properties 
that belongs to the same conjugacy class, and is, in a precise sense, 
no bigger than $g$. The resulting combinatorics seems to be new and 
may be of independent interest. 

The structure of the paper is as follows. Section~\ref{background} 
provides definitions and basic results about group representations and 
%combinatorics of 
Coxeter groups. Section~\ref{S:joint-spectra}
contains the needed background material and 
%reviews previously obtained 
results regarding joint spectra. %It also  

In Section~\ref{dihedral-group} we 
analyze joint spectra of involution pairs. The computations reveal an 
interesting connection to Tchebyshev's polynomials. 
%The second is that this technique 
We use these results and give the proof of 
Theorem~\ref{T:main-theorem-0}. 

In Section~\ref{dihedral2} we work 
with finite dimensional unitary representations of dihedral groups. 
We demonstrate how one can use the tools from \cite{ST} to construct 
explicitly the decomposition of such a representation into irreducible 
factors directly from the defining equation of the joint spectrum.   
We use that to give a proof of Theorem~\ref{main-theorem} for the 
case when $G$ is finite dihedral, and prove an analogous result for 
unitary representations of the infinite dihedral group. We also address question (3) for that setting. Of course, since the irreducible 
representations of dihedral groups are well known, these results 
can also be obtained by direct analysis of the determinantal 
hypersurfaces of these irreducible representations. 

Sections~\ref{c-length} and~\ref{S:echelon-forms} are devoted to 
developing the combinatorics needed for the proof of 
Theorem~\ref{main-theorem} for $G$ of types $A$, $B$, or $D$. 
Section~\ref{c-length} considers words in the alphabet of the 
Coxeter generators of $G$, and introduces a partial ordering 
and \emph{admissible transformations} on these words. We show 
that admissible transformations preserve conjugacy classes, and 
are non-increasing with respect to the partial ordering. 
Section~\ref{S:echelon-forms} introduces echelon forms. The main 
result there is the ``Ordering Theorem'', 
Theorem~\ref{T:ordering-theorem}, which states that every word 
can be transformed into echelon form via a sequence of admissible 
transformations. This result is the key to our proof of 
Theorem~\ref{main-theorem}, which we present      
in Section~\ref{main theorem}. 

%It turned out that determinantal hypersurfaces characterize 
%representations of some affine groups as well. 
The combinatorics of admissible transformations can be extended to 
other Coxeter groups. 
In Section~\ref{affine c} we give an example how 
this can be used to study representations of 
the affine Coxeter group 
%determinantal hypersurfaces %in divisor form 
%characterize linear representations of the affine group 
$\tilde{C}_2$ via determinantal hypersurfaces.

\section{%\bf Background results
Preliminaries
}\label{background}

%\vspace{.2cm}

%\subsection{\bf Representations of groups and their characters}

Let $G$ be group. Recall that a (complex) representation of $G$ 
is a group homomorphism $\rho: G\longrightarrow GL(V)$ from $G$ to 
the group of bounded invertible linear operators on a Hilbert space $V$. 
The representation $\rho$ is called \emph{unitary} 
if $\rho(w)$ is unitary for all $w\in G$; it is called 
\emph{finite dimensional} if $V$ is finite dimensional; and it is 
called \emph{faithful} if $\rho$ is injective. 
Two representations $\rho_1: G\rightarrow GL(V_1)$ and 
$\rho_2: G\rightarrow GL(V_2)$ are \emph{equivalent} 
if there is a bounded linear isomorphism $C: V_1\rightarrow V_2$ 
%of the space $V$ 
such that  \ $\rho_2(w)=C\rho_1(w)C^{-1}$ for every $w\in G$.

The group ring $\C[G]$ has inner product %given by 
$
\langle
\sum_{g\in G}a_g g, \sum_{g\in G}b_g g
\rangle 
=
\sum_{g\in G}a_g\bar b_g 
$ 
and corresponding induced norm 
$
\left\Vert
\sum a_g g 
\right\Vert 
=
\sum |a_g|^2. 
$
We write $\C[G]^{\vee}$ for the Hilbert space obtained 
by completing $\C[G]$ with respect to this norm. 
Of course, if $G$ is finite then $\C[G]^\vee=\C[G]$. 
For each $g\in G$ the left multiplication by $g$ on $\C[G]$ 
induces a bounded invertible linear operator $\lambda(g)$ 
on $\C[G]^\vee$, and the resulting map 
$\lambda: G\rightarrow GL\bigl(\C[G]^\vee\bigr)$ is a faithful 
unitary representation of $G$ called the \emph{left regular 
representation}.   
%, the representation $\rho$ is called unitary.

When dealing with representations of finite groups we will rely heavily 
on the following basic fact 
%about linear representations of finite groups is the following result 
(see \cite{S}). 

%\vspace{.2cm}

\begin{theorem}\label{unitary}
Every linear representation of a finite group $G$ is equivalent to a unitary representation.	
\end{theorem}

An important invariant of a finite dimensional 
representation $\rho$ of any group $G$ is the 
\textbf{character $\chi_\rho$ of} $\rho$. This is a function on $G$ 
defined by
$$
\chi_\rho(w)=\trace\bigl(\rho(w)\bigr).
$$
The character of $\rho$ is  a \emph{class-function}, that is, it is 
constant on conjugacy classes. In fact, when $G$ is finite 
it determines completely $\rho$
(see \cite{S}):  

%\vspace{.2cm}

\begin{theorem}\label{character determines representation}
Let $G$	be a finite group, and $\rho_1$ and $\rho_2$ be two 
finite dimensional representations of $G$. 
If $\chi_{\rho_1}=\chi_{\rho_2}$, then $\rho_1$ and $\rho_2$ are equivalent.
\end{theorem}

%\vspace{.2cm}
%
%Finally, it is well-known that for a finite group the number of irreducible 
%representations is equal to the number of conjugacy classes and that the 
%characters of irreducible representations separate conjugacy classes. 
%This, of course implies that $w_1$ and $w_2$ belong to the same conjugacy 
%class if and only if for every linear representation 
%$\rho$ of $G$, \ $\chi_\rho(w_1)=\chi_\rho(w_2).$ 

%\subsection{\bf Coxeter groups}
Next, we review some basics on Coxeter groups. 
The monographs \cite{BB}, \cite{GP}, and \cite{Hu} are good sources for 
their combinatorics, and properties.

A Coxeter group is a finitely generated group $G$ on generators 
$g_1,\dots,g_n$ defined by the following relations:
\[
%g_j^2=1, \quad 
(g_ig_j)^{m_{ij}}=1, \quad i,j=1,\dots,n,
\]
where $m_{ii}=1$ and $m_{ij}\in {\mathbb N}\cup \{ \infty \}$, %\ 
with $m_{ij}\geq 2$ when $i\neq j$. 
It is easy to see that to avoid redundancy we must have $m_{ij}=m_{ji}$, 
and that $m_{ij}=2$ means $g_i$ and $g_j$ commute. 
The set of generators $\{g_1,\dots, g_n\}$ is called a \emph{Coxeter set 
of generators}, and the $m_{ij}$s are called the \emph{Coxeter exponents}.   
%The symmetric matrix 
%$$M=\left [ m_{i,j}\right ]_{i,j=1}^n $$
%($m_{ii}=1$), which is called the Coxeter matrix, determines the group.

A traditional way of presentation of a Coxeter group is through its  
\emph{Coxeter diagram}, which is a graph constructed by the 
following rules:
\begin{itemize}
\item 
the vertices of the graph are the generator subscripts;

\item 
vertices $i$ and $j$ form an edge if and only if $m_{ij}\geq 3$;

\item 
an edge is labeled with the value $m_{ij}$ whenever this value 
is $4$ or greater.
\end{itemize}
In particular, two generators commute if and only if they are not connected by an edge. The disjoint union of Coxeter diagrams yields a direct product of Coxeter groups, and a Coxeter group is \emph{connected} if its diagram  
is a connected graph. 

The finite connected 
Coxeter groups consist of the one-parameter families $A_n$, 
$B_n$, $D_n$, and  $I(n)$, and the six exceptional groups 
$E_6$, $E_7$, $E_8$, $F_4$, $H_3$, and $H_4$. 
They were classified by Coxeter~\cite{C2}, and in \cite{C1} 
Coxeter proved that every reflection group is a Coxeter group. 

The Coxeter diagrams for the groups $A_n, B_n, D_{n+1}$, and $I(n)$ that 
we study in this paper are as follows:

\vspace{.2cm}

%$A_n$:
%\begin{center}
\begin{tikzpicture}[scale=1, vertices/.style={draw, fill=black,
                            circle, inner sep=1pt}]
              \node [label=left:{$\quad A_n:$}]              (8) 
                                                  at (-5.25, .25){};
              \node [vertices, label=below:{$1$}]            (1) 
                                                  at (-4.25, .25){};
              \node [label=above:{$\vphantom{2}$}]           (2) 
                                                  at (-3.50, .25){};
              \node [vertices, label=below:{$2$}]            (3) 
                                                  at (-2.75, .25){};
              \node [label=right:{$\quad\! \ldots$}]         (4) 
                                                  at (-1.25, .25){};
              \node                                          (5) 
                                                  at (  .25, .25){};
	      \node [vertices, label=below:{$n-1$}]          (6) 
                                                  at ( 1.75, .25){};
              \node [vertices, label=below:{$\vphantom{1}n$}](7) 
                                                  at ( 3.25, .25){};

	   \foreach \to/\from in {1/3, 3/4, 5/6, 6/7}
	   \draw [-] (\to)--(\from);
	   %\draw [-] (-7,0) -- (-6,-1);
\end{tikzpicture} 
%\newline 
%\end{center}

\vspace{.2cm}

%$B_n$:
\begin{tikzpicture}[scale=1, vertices/.style={draw, fill=black,
                            circle, inner sep=1pt}]
              \node [label=left:{$\quad B_n:$}]              (8) 
                                                  at (-5.25, .25){};
              \node [vertices, label=below:{$1$}]            (1) 
                                                  at (-4.25, .25){};
              \node [label=above:{\small{4}}]              (2) 
                                                  at ( 2.50, .25){};
              \node [vertices, label=below:{$2$}]            (3) 
                                                  at (-2.75, .25){};
              \node [label=right:{$\quad\! \ldots$}]         (4) 
                                                  at (-1.25, .25){};
              \node                                          (5) 
                                                  at (  .25, .25){};
	      \node [vertices, label=below:{$n-1$}]          (6) 
                                                  at ( 1.75, .25){};
              \node [vertices, label=below:{$\vphantom{1}n$}](7) 
                                                  at ( 3.25, .25){};
%              \node [vertices, label=right:{$n+1$}]          (8)
%                                                  at ( 3.25,-.25){}; 

	   \foreach \to/\from in {1/3, 3/4, 5/6, 6/7}
	   \draw [-] (\to)--(\from);
	   %\draw [-] (-7,0) -- (-6,-1);
\end{tikzpicture}

\vspace{.2cm}

%$D_n$:
\begin{tikzpicture}[scale=1, vertices/.style={draw, fill=black,
                            circle, inner sep=1pt}]
              \node [label=left:{$D_{n+1}:$}]                 (9) 
                                                  at (-5.25, .25){};
              \node [vertices, label=below:{$1$}]            (1) 
                                                  at (-4.25, .25){};
              \node [label=above:{$\vphantom{4}$}]           (2) 
                                                  at ( 2.50, .25){};
              \node [vertices, label=below:{$2$}]            (3) 
                                                  at (-2.75, .25){};
              \node [label=right:{$\quad\! \ldots$}]         (4) 
                                                  at (-1.25, .25){};
              \node                                          (5) 
                                                  at (  .25, .25){};
	      \node [vertices, label=below:{$n-1$}]          (6) 
                                                  at ( 1.75, .25){};
              \node [vertices, label=right:{$\vphantom{1}n$}](7) 
                                                  at ( 3.25, .75){};
              \node [vertices, label=right:{$n+1$}]          (8) 
                                                  at ( 3.25,-.25){}; 
	   \foreach \to/\from in {1/3, 3/4, 5/6, 6/7, 6/8}
	   \draw [-] (\to)--(\from);
	   %\draw [-] (-7,0) -- (-6,-1);
\end{tikzpicture} 

\vspace{.2cm}

%$I(n)$:
\begin{tikzpicture}[scale=1, vertices/.style={draw, fill=black,
                            circle, inner sep=1pt}]
              \node [label=left:{$\ \, I(n):$}]              (9) 
                                                  at (-5.25, .25){};
              \node [vertices, label=below:{$1$}]            (1) 
                                                  at (-4.25, .25){};
              \node [label=above:{$n$}]                      (2) 
                                                  at (-3.50, .25){};
              \node [vertices, label=below:{$2$}]            (3) 
                                                  at (-2.75, .25){};
              %\node [label=right:{$\quad\! \ldots$}]         (4) 
              %                                    at (-1.25, .25){};
              %\node                                          (5) 
              %                                    at (  .25, .25){};
	      %\node [vertices, label=below:{$n-1$}]          (6) 
              %                                    at ( 1.75, .25){};
              %\node [vertices, label=right:{$\vphantom{1}n$}](7) 
              %                                    at ( 3.25, .75){};
              %\node [vertices, label=right:{$n+1$}]          (8) 
              %                                    at ( 3.25,-.25){}; 
	   %\foreach \to/\from in {1/3, 3/4, 5/6, 6/7, 6/8}
	   \draw [-] (1)--(3);
	   %\draw [-] (-7,0) -- (-6,-1);
\end{tikzpicture}

%We conclude this section by observing 
%that $A_2\cong I(3)$, that $B_2\cong I(4)$, and that $A_3\cong D_3$.   

\section{More on joint spectra}\label{S:joint-spectra}
 
Recall that for an algebraic hypersurface in $\C\p^n$ or in $\C^n$ 
defined by
a polynomial $F=F_1^{r_1}\dots F_m^{r_m}$
(with each polynomial $F_i$ irreducible and $F_i$ not associate
with $F_j$ for $i\ne j$),
the \emph{components} of that hypersurface are defined by the
polynomials $F_i^{r_i}$ (thus they are irreducible
but not necessarily reduced), the
\emph{reduced components} are defined by the
polynomials $F_i$, and the exponent $r_i$ is called the 
\emph{multiplicity} of the reduced component defined by $F_i$.
We say that a point on our hypersurface is \emph{regular (with 
multiplicity $r$)} if it belongs to only one reduced component  
(of multiplicity $r$), and  
is a regular point on that reduced component. 
%and this component has multiplicity $r$. 

Let $A$ be a bounded linear operator on a Hilbert space 
$V$. We consider its norm given by 
$\left\Vert A\right\Vert=\sup_{|u|=1}|A(u)|$.  
The \emph{spectrum} of $A$ is the set 
\[
\sigma(A)=\{\lambda\in\C\mid A-\lambda I\text{ is not invertible}\}.
\]  
Note that when $A$ is self-adjoint we have 
$\sigma(A)\subseteq\mathbb R$ and 
$\left\Vert A\right\Vert =\sup_{\lambda\in\sigma(A)}|\lambda|$. 
We will use the following elementary consequences of the 
Spectral Theorem and the Spectral Mapping Theorem, see \cite{L}. 

\begin{proposition}\label{P:spectral-facts}
Let $A$ be a bounded linear operator on a Hilbert space $V$. 

(a) 
If $p(x)$ is any polynomial %with complex coefficients 
then 
$
\sigma\bigl(p(A)\bigr)=p\bigl(\sigma(A)\bigr)= 
\{ p(\lambda) \mid \lambda\in\sigma(A)\}.  
$

(b) 
If $A$ is normal and $\sigma(A)=\{1\}$ then $A=I$, the identity 
operator. 
\end{proposition}

Now let $A_1,A_2,\dots,A_n$  be %$N\times N$ complex matrices. 
bounded linear operators on a Hilbert space $V$. 
We will be considering the projective joint spectrum 
$\sigma(-I,A_1,\dots,A_n)$ of the tuple  
$-I, A_1, A_2,\dots,A_n$, where $I$ is the identity operator, and 
its intersection with the %affine 
chart $\{[x_0:\dots : x_n] \mid x_0\ne 0\}$. 
By taking $x_0=1$ we identify this intersection with 
a closed subset of 
%an affine closed  subscheme of 
$\C^n$ called the \emph{proper joint spectrum} of  the tuple 
$A_1,\dots, A_n$, denoted by $\sigma_p(A_1,\dots, A_n)$; and 
when $V$ is finite dimensional 
its defining polynomial is  
%have  
$$
%\sigma_p(A_1,\dots,A_n)=
\mathcal F(x_1,\dots,x_n)= %\in \C^n \mid 
\det(-I +x_1A_1+\dots+x_nA_n). %= 0 \}.
$$ 
In particular, a point $x=(x_1,\dots,x_n)\in\C^n$ belongs to the %(reduced) 
proper joint spectrum if and only if 
the operator $A(x)=x_1A_1+\dots+x_nA_n$ has $1$ in its spectrum. 
We note that %as long as 
when $V$ is finite dimensional and 
$\sigma_p(A_1,\dots, A_n)$ is not empty  
(this latter is the case, for example, if $A_i$ is invertible for some $i$), 
the projective joint spectrum $\sigma(-I, A_1,\dots, A_n)$ is the 
closure of the determinantal hypersurface 
$\sigma_p(A_1,\dots, A_n)$, and its defining polynomial is 
the homogenization of $\mathcal F(x_1,\dots, x_n)$.  

We will be using 
the following basic observation multiple times:

\begin{proposition}\label{R:multiplicity-eigenvalues}
%It is a standard exercise that when $A_1$ and $A_2$ are matrices 
Suppose $V$ is finite dimensional, 
let $u=(u_1,u_2)$ be a regular point of $\sigma_p(A_1,A_2)$ of 
multiplicity~$r$, and let $\Gamma$ be the corresponding 
unique reduced component containing $u$. 

If the tangent line to $\Gamma$ at $u$ does not pass through the 
origin $(0,0)$ then 
%at every nonsingular point $(x,y)$ of $\Gamma$ 
the operator $A(u)=u_1A_1 + u_2A_2$ has eigenvalue $1$ of 
multiplicity exactly $r$.  
\end{proposition}

\begin{proof} 
Indeed, let 
$
\mathcal F(x_1,x_2)=\det(-I + x_1A_1 + x_2 A_2)= 
F^r H
$ 
be the defining polynomial of $\sigma_p(A_1, A_2)\ne \emptyset$, 
where $F$ is a defining irreducible polynomial of $\Gamma$, and 
$H$ is not divisible by $F$. Then 
the defining polynomial 
$
\mathcal P(x_0,x_1,x_2)=
\det(-x_0I + x_1A_1 + x_2A_2)=
\tilde F^r\tilde H
$
of $\sigma(-I, A_1, A_2)$ is the homogenization of $\mathcal F$, 
where $\tilde F$ and $\tilde H$ are the homogenizations of $F$ and $H$. 
Note $\mathcal P$  
is also (as a polynomial in $x_0$) the characteristic polynomial of 
the operator $A(x)$. Thus we have 
$
\mathcal P(x_0,u_1,u_2)=
\tilde F^r(x_0,u_1,u_2)\tilde H(x_0,u_1,u_2) = 
(1-x_0)^m\prod_{i=1}^k (\lambda_i- x_0)^{m_i}
$
where the $\lambda_i$ are the eigenvalues of $A(u)$ different from $1$, 
and the $m_i$ their multiplicities. Since $u$ is only on $\Gamma$, it 
follows that $\tilde H(1,u_1,u_2)=H(u)\ne 0$, hence $(1-x_0)^m$ is a 
factor of $\tilde F^r(x_0,u_1,u_2)$. Thus it suffices to check that 
$x_0=1$ is not a double root of $\tilde F(x_0,u_1,u_2)$ when considered 
as a polynomial in $x_0$. Since for 
the corresponding partial derivative in $x_0$  
we have 
$
\tilde F_{x_0}(1,u_1,u_2)=
- u_1F_{x_1}(u_1,u_2) - u_2F_{x_2}(u_1,u_2),  
$
and the equation for the tangent line to $\Gamma$ at $u$ is 
$(x_1-u_1)F_{x_1}(u) + (x_2-u_2)F_{x_2}(u)=0$, 
the desired conclusion is immedaite from our assumption that the 
tangent line does not contain the origin. 
\end{proof}

It was shown in \cite{ST} that substantial information regarding mutual relations between the $A_i$s is captured by the geometry of their proper   
joint spectrum, and we review some of these results below, %focusing on 
in the case of two operators $A_1$ and $A_2$.

%We start with an elementary result from \cite{CSZ} 
%about spectra of linear transforms. 
First we note the behaviour of the joint spectrum under change 
of coordinates.  
Let 
\[ 
{\bf C}= 
\left[ 
\begin{array}{ll} 
   c_{11} & c_{12} \\ 
  c_{21} & c_{22} 
\end{array} 
\right ] 
\] 
be an invertible  complex-valued matrix. Write 
\begin{equation}\label{transform} 
B_1= c_{11}A_1+c_{12}A_2, \quad  B_2=c_{21}A_1+c_{22}A_2. 
\end{equation} 
Then it is immediate from the definitions that  
\begin{equation}\label{spectransform} 
\sigma_p(A_1,A_2)={\bf C}^T \sigma_p(B_1,B_2). 
\end{equation}

%Now consider the case of two operators $A_1$ and $A_2$. 
Next, suppose $V$ is finite dimensional and $A_1$ is self-adjoint.   
Let $\lambda \neq 0$ be an eigenvalue of $A_1$, %of finite multiplicity, 
and  suppose that $(1/\lambda,0)$ is a regular point of 
the determinantal hypersurface $\sigma_p(A_1,A_2)$.  
%with multiplicty $r$. 
Let $\mathcal R(x_1,x_2)$ be a defining polynomial 
for the reduced component of $\sigma_p(A_1,A_2)$ that contains 
$(1/\lambda, 0)$; thus 
%where ${\mathcal R}$ is a polynomial. Let us also assume so that 
one of the partial derivatives, say, 
$\mathcal R_{x_1}=\frac{\partial {\mathcal R}}{\partial x_1}$,  
does not vanish at $(1/\lambda,0)$. Then  
%so $(1/\lambda,0)$ is a regular point of $\sigma_p(A_1,A_2)$. 
the implicit function theorem implies that in a neighborhood of 
$(1/\lambda,0)$ the equation $\{ {\mathcal R}=0 \}$ determines $x_1$ as an implicit analytic function of $x_2$ defined in a neighborhood of the origin:
\begin{equation}\label{implicit}
x_1=x_1(x_2), \quad x_1(0)=\frac{1}{\lambda}.
\end{equation}
For each eigenvalue $w$ in the spectrum $\sigma(A_1)$ of $A_1$ 
let $P_w$ be the 
orthogonal projection onto the $w$-eigensubspace of $A_1$. We also 
set $P=P_\lambda$.  
Since $A_1$ is self-adjoint, the spectral decomposition for $A_1$ is  
\begin{equation}\label{spectral resolution}
A_1\quad =\quad \lambda P\quad + \sum_{w\in\sigma(A_1)\setminus\{\lambda\}} w P_w.
\end{equation}
We also introduce the following operator:  
\begin{equation}\label{t}
T(A_1) \quad = \quad T\quad =
\sum_{w\in\sigma(A_1)\setminus \{ \lambda \}} \frac{\lambda}{w-\lambda}P_w. 
\end{equation}

The following result was proved (in significantly greater generality) 
during the course of \cite[Proof of Theorem 7.3]{ST}.

\begin{theorem}\label{algebraic curve}
Suppose that $V$ is finite dimensional, 
that $A_1$ and $A_2$ are self-adjoint, and that $\lambda\neq 0$ is an 
eigenvalue of $A_1$ %of finite multiplicity 
such that
\begin{itemize}
\item 
$(1/\lambda,0)$	is a regular point of $\sigma_p(A_1,A_2)$; and  

\item 
%in a neighborhood of $(1/\lambda,0)$ we have 
%the proper joint spectrum $\sigma_p(A_1,A_2)$ is given by 
%
%\item 
$
%\left. \frac{\partial {\mathcal R}}{\partial x_1}\right |_{(1/\lambda,0)}\neq 0 
\mathcal R_{x_1}(1/\lambda, 0)\neq 0, 
$ 
where ${\mathcal R}(x_1,x_2)$ is a defining polynomial for the 
reduced component of $\sigma_p(A_1,A_2)$ containing $(1/\lambda, 0)$.  
\end{itemize}
Then we have 
\begin{align}
PA_2P     &= -x_1^\prime (0)	P \label{algebraic curve1} \\
\intertext{and}  
PA_2TA_2P &= -\frac{x_1^{\prime \prime}(0)}{2} P, \label{algebraic curve2}
\end{align}
where we regard  $x_1$ as an implicit analytic 
function \eqref{implicit} of $x_2$ defined via the equation 
$\mathcal R=0$.
\end{theorem}

\section{Joint spectra and Coxeter groups}
\label{dihedral-group} 
%\label{Coxeter} 
 
Here we study properties of joint spectra of involution pairs 
$A_1, A_2$, and use these to prove 
Theorem~\ref{T:main-theorem-0}.

\begin{lemma}\label{dihedral} 
Let $A_1$ and $A_2$ be 
%unitary self-adjoint $N\times N$ complex matrices. 
bounded linear operators on a Hilbert space $V$ such that 
$A_1^2=A_2^2=I$.  Then: 
\begin{enumerate}
\item 
The set $\sigma_p(A_1,A_2)\cup -\sigma_p(A_1,A_2)$ is the union of 
%some lines of the form $x\pm y=\pm 1$, and 
all the ``complex ellipses" 
$\mathcal E_\alpha= \{x^2+ \alpha xy +y^2=1\}$ with 
$\alpha\in\sigma(A_1A_2 + A_2A_1)$. 
%$0<\theta < 1/2$. 

\item
When $\sigma(A_1A_2+A_2A_1)$ is a finite set then each connected 
component of $\sigma_p(A_1,A_2)\setminus \{(\pm 1,0) \ (0,\pm 1) \}$ 
is either $L\setminus \{(\pm 1,0) \ (0,\pm 1) \}$ with $L$ one of the 
lines $x\pm y=\pm 1$, or 
$\mathcal E_\alpha\setminus \{(\pm 1,0) \ (0,\pm 1) \}$ for some 
$\alpha\in\sigma(A_1A_2+A_2A_1)$.
   
\item
When $V$ is finite dimensional 
each reduced component of $\sigma_p(A_1,A_2)$ is either a line
of the form $x\pm y=\pm 1$, or a ``complex ellipse" $\mathcal E_\alpha$ 
with $\alpha\in\sigma(A_1A_2+A_2A_1)\setminus\{-2,2\}$. 
%$x^2+ xy\cos(2\pi\theta)+y^2=1$, where $0<\theta < 1/2$.
\end{enumerate}
\end{lemma}

\begin{proof} 
Suppose that $(x,y)\in \sigma_p(A_1,A_2)\cup -\sigma_p(A_1,A_2)$. 
%and that$\xi(x,y)$ is an eigenvector so that
This is if and only if 
 $1\in\sigma(xA_1+yA_2)$ or $-1\in\sigma_p(xA_1+yA_2)$, which,  
by Proposition~\ref{P:spectral-facts}, happens precisely when  
$1\in\sigma\bigl( (xA_1+yA_2)^2\bigr)$. That in turn happens if and only if 
the operator 
\begin{equation}\label{ellipse2}
(xA_1+yA_2)^2 - I = xy[A_1A_2+A_2A_1] - (1-x^2-y^2)I 
\end{equation}
is not invertible. We see that when $xy\ne 0$ this is 
equivalent to  
$\alpha=\frac{1-x^2-y^2}{xy}$ being in the spectrum of $A_1A_2+A_2A_1$, 
and $(x,y)$ being a point on $\mathcal E_\alpha$ (and this last fact is  
true for trivial reasons also when $xy=0$).  This completes the 
proof of part (1). 

When %$V$ is finite dimensional, 
$\sigma(A_1A_2+A_2A_1)$ is a finite set 
%. Therefore 
%If $(x,y)\neq (\pm 1,0) \ \mbox{or} \ (0,\pm 1),$ 
equation \eqref{ellipse2} implies that $\frac{1-x^2-y^2}{xy}$ 
is constant on every connected component of 
$\sigma_p(A_1,A_2)\setminus \{(\pm 1,0) \ (0,\pm 1) \} $, which yields 
(2). Therefore, if $V$ is finite dimensional,
then each reduced component of $\sigma_p(A_1,A_2)$ is %contained in 
a reduced component of an ellipse $\mathcal E_\alpha$  
%``complex ellipse" $\{\frac{1-x^2-y^2}{xy}=2\alpha\} $ 
for some $\alpha\in\sigma(A_1A_2+A_2A_1)$. 
When $\alpha\neq \pm 2$ the ellipse is irreducible, so 
our reduced component coincides with the ellipse. 
When $\alpha=\pm 2$ 
the ellipse is the union of two of the lines $x\pm y=\pm 1$, 
hence our reduced component is one of these lines.   
\end{proof} 
%\vspace{.2cm}

\begin{lemma}\label{L:dihedral-exponents} 
Let $A_1$ and $A_2$ %, and $\theta$ 
be as in Lemma~\ref{dihedral},  
and let $m\ge 2$ be an integer. 
%The following are equivalent: 
If 
\begin{enumerate}
\item
%if and only if 
$\ (A_1A_2)^m=I$  
\end{enumerate}
then 
\begin{enumerate} 
\item[(2)]
$
\ \sigma(A_1A_2 + A_2A_1)\ \subseteq \ 
\{\, 2\cos(2\pi k/m)\ \mid \ k=0,\dots, m-1\, \}$. 
%such that $\theta = k/m$; 
\end{enumerate} 
If both $A_1$ and $A_2$ are unitary, then conditions $(1)$ and $(2)$ 
are equivalent. 
\end{lemma}

\begin{proof} 
For each $n\ge 0$ set $R_n=(1/2)\bigl[ (A_1A_2)^n + (A_2A_1)^n\bigr]$. 
It is straightforward to check that  
\begin{align*}
R_0 &= I,  \\ 
R_1 &= (1/2)(A_1A_2+A_2A_1), \quad\text{and }  \\ 
%\big\vert_{L}=2\alpha I_{L}. 
R_n &= 2R_1 R_{n-1} - R_{n-2} \qquad\text{for } n\ge 2. 
\end{align*}
It follows immediately by induction 
%We can inductively show 
that for each $n\ge 0$ we have 
$$
R_n=T_n(R_1), 
$$
where $T_n(z)$ are Tchebyshev's polynomials of the first kind defined by
\begin{align*}
T_0(z)   &=1, \\
T_1(z)   &=z, \quad\text{and }\\
T_n(z)   &=2z T_{n-1}(z)-T_{n-2}(z) \qquad\text{for } n\ge 2. 
\end{align*}
It is well known that for each real $z\in [-1,1]$ one has 
$T_n(z)=\cos (n \cos^{-1}(z))$, cf. Rivlin~\cite{Ri}, in particular  
the polynomial $T_n(z) - 1$ is of degree $n$ and has for its set of 
roots the set $\{\, \cos(2\pi k/n)\ \mid \ k=0,\dots n-1\, \}$.   

Now, suppose $(A_1A_2)^m=I$. Thus $(A_2A_1)^m=I$ as well, hence 
$R_m=T_m(R_1)=I$. Since by Proposition~\ref{P:spectral-facts}(a) we have 
$\sigma(R_m)=T_m\bigl(\sigma(R_1)\bigr)$, we must have $T_m(\alpha)=1$ 
for each $\alpha\in\sigma(R_1)$. Therefore 
$\sigma(R_1)\subseteq\{\cos(2\pi k/m) \mid  k=0,\dots, m-1 \}$, 
which implies $(2)$ as desired. 
%Since $(A_1A_2)^m=(A_2A_1)^m=I$ we have $R_m=2I_{L_\alpha}$, and, hence,

Next, suppose $A_1$ and $A_2$ are both unitary and $(2)$ holds. 
Then $R_n$ is self-adjoint for each $n$. In particular $R_m$ is 
self-adjoint and %by Proposition~\ref{P:spectral-facts}(a) 
has $\sigma(R_m)=T_m\bigl(\sigma(R_1)\bigr)=\{1\}$.  
It follows from Proposition~\ref{P:spectral-facts}(b) that $R_m=I$. 
Thus $(A_1A_2)^m + (A_1A_2)^{-m}=2I$ hence $(A_1A_2)^{2m}-2(A_1A_2)^m + I=0$. 
Proposition~\ref{P:spectral-facts}(a) now yields that $(\lambda -1)^2=0$ 
for every $\lambda\in\sigma\bigl( (A_1A_2)^m\bigr)$. So 
$(A_1A_2)^m$ is a unitary hence normal operator whose spectrum is the 
singleton $\{1\}$, and  therefore $(1)$ holds by 
Proposition~\ref{P:spectral-facts}(b). 
\end{proof}

\begin{remark} 
The above proofs show that the spectral decomposition of the operator 
$A_1A_2+A_2A_1$ determines
the decomposition of the pair $(A_1,A_2)$ and the decomposition of the joint 
spectrum of $A_1$ and $A_2$ into
irreducible components. The fact that $A_1A_2+A_2A_1$ commutes with both $A_1$ 
and $A_2$, and, therefore,
belongs to the center of the group-algebra, plays a key role here.
\end{remark}

%\vspace{.2cm}
Now we are ready to present the proof of Theorem~\ref{T:main-theorem-0}. 

\begin{proof}[Proof of Theorem~\ref{T:main-theorem-0}]
We have $T=\{1_G,g_1,\dots,g_n\}$ where $\{g_1,\dots, g_n\}$ is 
a system of Coxeter generators of for $G$. Also, we have 
$T'=\{1_{G'}, g_1',\dots,g_n'\}$ where $\{g_1',\dots,g_n' \}$ is a 
generating set for the group $G'$. Let $A_i$ (resp. $A_i'$) 
be the operator on $\C[G]^\vee$ (resp. $\C[G']^\vee$) induced 
%given 
via left multiplictaion by $g_i$ 
(resp. $g_i'$) on the group ring $\C[G]$ (resp. $\C[G']$). %and let   
Suppose $D(T)\supseteq D(T')$. 
%hence $|G|=\deg D(T) =\deg D(T')= |G'|$. 
We also have 
$\sigma_p(A_1,\dots, A_n)\supseteq \sigma_p(A_1',\dots,A_n')$. Furthermore, 
for all $i$ and all $k<l$ 
intersecting with the projective line 
$\{ x_j=0 \mid 0\ne j\ne i\}$, and with the projective plane 
$\bigl\{ x_j=0 \mid j\notin\{0,k,l\}\bigr\}$ 
%$j$th, and $k$th coordinates 
shows  that 
$\sigma(-I,A_i)\supseteq \sigma(-I,A_i')$ and  
$\sigma(-I,A_i,A_j)\supseteq \sigma(-I,A_i',A_j')$, hence also that
$\sigma_p(A_i,A_j)\supseteq \sigma_p(A_i', A_j')$. 
In particular, for each $i$ we have 
$\{-1,1\}=\sigma(A_i)\supseteq\sigma(A_i')$.  
%have the same eigenvalues, namely $\pm 1$. 
Since left regular representations are unitary, 
$A_i'$ is normal and hence an involution by
Proposition~\ref{P:spectral-facts}.  

Next, suppose that the Coxeter exponent $m_{ij}$ is finite. Then 
by Lemma~\ref{L:dihedral-exponents} the spectrum of $A_iA_j+A_jA_i$ 
is finite and contained in the set 
$\{ 2\cos(2\pi k/m_{ij}) \mid k=0,\dots,m_{ij}-1 \}$. 
Hence by Lemma~\ref{dihedral} it is 
completely determined by  $\sigma_p(A_i,A_j)$, which is a finite union 
of ellipses. Since $\sigma_p(A_i',A_j')$ 
is contained inside $\sigma_p(A_i,A_j)$, again Lemma~\ref{dihedral}  
tells us that the spectrum of $A_i'A_j'+A_j'A_i'$ is finite
and contained in the spectrum of $A_iA_j + A_jA_i$. Therefore by   
Lemma~\ref{L:dihedral-exponents} we obtain $(A_1'A_2')^{m_{ij}}=I$. 
Since the left regular representation is faithful, the 
%Thus the 
generators $g_i'$ of $G'$ satisfy the defining relations on 
the generators $g_i$ of $G$. 

Now suppose $G$ is finite and $D(T)=D(T')$ as subschemes of $\C\p^n$.
But then $|G|=\deg D(T) =\deg D(T')= |G'|$, hence the epimorphism $f$ 
has to be an isomorphism.   
\end{proof}

%\vspace{.2cm}

\section{Representations of dihedral groups}
\label{dihedral2}

We investigate finite dimensional unitary representations of 
dihedral groups. 

The following result was for us the first indication that joint 
spectra can be employed to study Coxeter groups. 
Observe how the proof uses the equation of the circle to explicitly 
construct the desired invariant subspace.    

\begin{theorem}\label{circle}
Let $A_1,A_2$ be self-adjoint linear operators on an 
$N$-dimensional Hilbert space $V$, and suppose that $A_1$ is
invertible and that $\left\Vert A_2\right\Vert =1$.  
%where the norm is understood 
%as the norm of a self-map of the Euclidean space $\C^N$.
Further suppose that the ``complex unit circle" 
$\{(x,y)\in  \C^2: x^2+y^2=1\}$ is a reduced 
component of both $\sigma_p(A_1,A_2)$ and 
$\sigma_p(A_1^{-1},A_2)$, 
 of multiplicity $n$ in $\sigma_p(A_1,A_2)$, 
that the points $(\pm 1,0)$ do not belong 
to any other component of either $\sigma_p(A_1,A_2)$ or 
$\sigma_p(A_1^{-1},A_2)$, and that the points  
$(0,\pm 1)$ do not belong to any other component of 
$\sigma_p(A_1,A_2)$. 
Then: 
\begin{enumerate}
\item%[(1)]
$A_1$ and $A_2$ have a common $2n$-dimensional invariant 
subspace $L$;  
%which is a direct sum of $n$ two-dimensional 
%reducing subspaces, so that

\item
The pair of restrictions  $A_1|_L$ and $A_2|_L$ is
unitary equivalent to the following pair of $2n\times 2n$ involutions
$C_1$ and $C_2$, each block-diagonal with $n$ equal 
$2\times 2$ blocks along the diagonal: 
$$
C_1=
\left[
\begin{array}{ccccc}%cc}
                  1 & 0      & \dots  & 0 & 0  \\
                  0 & -1     & \dots  & 0 & 0 \\
             \vdots & \vdots & \ddots & \vdots & \vdots \\
                  0 & 0      & \dots  & 1 & 0\\
%                  .. & .. & .. & .. & ... & ... & .. \\
%                  0 & 0 & 0 & 0 & ... & 1 & 0 \\
                  0 & 0      & \dots  & 0 & -1
\end{array}
\right], \
C_2=
\left[
\begin{array}{ccccc}%cc}
                   0 & 1 & \dots & 0 & 0  \\
                  1  & 0 & \dots & 0 & 0 \\
%                  0 & 0 & 0 & 1 & ... & ... & 0 \\
%                  0 & 0 & 1 & 0 & ... & ... & 0\\
             \vdots & \vdots & \ddots & \vdots & \vdots \\
                  0 & 0 & \dots & 0 & 1 \\
                  0 & 0 & \dots & 1 & 0
\end{array}
\right].
$$

\item%[(2)]
The group generated by $C_1$ and $C_2$ represents
the Coxeter group $B_2$.
\end{enumerate}
\end{theorem}

\begin{proof}
Since the coordinate axes are not tangent to the unit circle at any  
of the points $(\pm 1,0)$ and $(0,\pm 1)$, and these are regular
points of the unit circle, Proposition~\ref{R:multiplicity-eigenvalues} 
shows that both $1$ and $-1$ are eigenvalues of $A_1$ and of $A_2$ 
with multiplicity $n$.  
Let  $e_1,\dots,e_{n}$ and $e_{n+1},\dots, e_{2n}$ be a pair
of  orthonormal  eigenbases for the eigenspaces of $A_1$
with eigenvalues $1$ and $-1$ respectively, and
$\xi_1,\dots,\xi_{n}$ and $\xi_{n+1},\dots,\xi_{2n}$ be a
similar pair of orthonormal eigenbases for $A_2$.
The spectral decomposition of $A_1$ in our case looks like
$$
A_1=P_1-P_2+\sum_{j=3}^k \lambda_jP_j,
$$
where $P_1$ and $P_2$ are the orthogonal projections on
the spaces $span\{ e_1,\dots,e_{n}\}$ and
$span \{ e_{n+1},\dots,e_{2n}\}$ respectively, 
$\lambda_j$ for $j=3,\dots,k$ are all other eigenvalues of $A_1$, 
and $P_j$ is the orthogonal projection on the 
$\lambda_j$-eigensubspace of $A_1$. In our situation equations 
\eqref{algebraic curve1} and \eqref{algebraic curve2} 
applied to the projection $P_1$ give
\begin{eqnarray}
P_1A_2P_1=0 \label{circle_order_one} \\
P_1A_2TA_2P_1= - \frac{1}{2}P_1 \label{circle_order_2},
\end{eqnarray}
where $T$ is given by \eqref{t}, which applied to our case turns into
$$
T=-\frac{1}{2} P_2+\sum_{j=3}^k\frac{1}{\lambda_j-1}P_j.
$$
Equations \eqref{circle_order_one} and \eqref{circle_order_2} 
imply that for
every $j=1,\dots,n$ we have
%\[
%\displaybreak
\begin{align*}
P_1A_2TA_2e_j &=-\frac{1}{2}e_j, %\nonumber 
\\
-P_1 A_2 
\left( 
\frac{1}{2}P_2A_2e_j + 
\sum_{s=3}^k \frac{1}{1-\lambda_s}P_sA_2e_j 
\right) 
&= -\frac{1}{2}e_j, %\nonumber 
%\displaybreak 
\\
P_1 
\left( 
\frac{1}{2}A_2P_2A_2e_j + 
\sum_{s=3}^k 
\frac{1}{1-\lambda_s}A_2P_sA_2e_j 
\right) 
&= \frac{1}{2}e_j, %\nonumber 
%\displaybreak 
\\
\sum_{l=1}^n 
\left[ 
\frac{1}{2}\langle A_2P_2A_2e_j, e_l\rangle + 
\sum_{s=3}^k 
\frac{1}{1-\lambda_s}\langle A_2P_sA_2e_j,e_l \rangle 
\right] 
e_l 
&= \frac{1}{2}e_j , %\nonumber 
%\displaybreak
\\
\sum_{l=1}^n 
\left[ 
\frac{1}{2}\langle P_2A_2e_j, A_2e_l\rangle + 
\sum_{s=3}^k 
\frac{1}{1-\lambda_s}\langle P_sA_2e_j,A_2e_l \rangle 
\right] 
e_l
&= \frac{1}{2}e_j, %\nonumber 
\\
\sum_{l=1}^n 
\left[ 
\frac{1}{2}\langle P_2A_2e_j, P_2A_2e_l\rangle + 
\sum_{s=3}^k 
\frac{1}{1-\lambda_s}\langle P_sA_2e_j,P_sA_2e_l \rangle 
\right] 
e_l
&= \frac{1}{2}e_j , %\nonumber
\end{align*}
%\]
The last relation implies
\begin{equation}\label{coordinates}
\frac{1}{2}\langle P_2A_2e_j, P_2A_2e_l\rangle +
\sum_{s=3}^k \frac{1}{1-\lambda_s}\langle P_sA_2e_j,P_sA_2e_l \rangle =
\frac{1}{2}\delta_j^l 
\end{equation}
for all $1\leq j,l\leq n$, 
where $\delta_j^l$ is the Kronecker symbol. Putting $l=j$ in \eqref{coordinates} yields
\begin{equation}\label{squares_for A_1}
\frac{1}{2}\left|P_2A_2e_j\right|^2 +
\sum_{s=3}^k\frac{\left|P_s(A_2e_j)\right|^2}{1-\lambda_s}
=
\frac{1}{2}, \qquad 1 \leq j\leq n.
\end{equation}
Since the spectral resolution for $A_1^{-1}$ is
$$ 
A_1^{-1}=P_1-P_2+ \sum_{j=3}^k \frac{1}{\lambda_j}P_j, 
$$
the operator $T(A_1^{-1})$ is given by
$$ 
T(A_1^{-1})= -\frac{1}{2}P_2 + \sum_{j=3}^k \frac{\lambda_j}{1-\lambda_j}P_j.
$$
Now, equation \eqref{algebraic curve2} applied to
$\sigma_p(A_1^{-1},A_2)$ yields
\begin{equation}\label{circle inverse}
P_1A_2T(A_1^{-1})A_2P_1=-\frac{1}{2}P_1,
\end{equation}
A similar argument applied to \eqref{circle inverse} gives
\begin{equation}\label{squares_for_inverse}
\frac{1}{2}\left| P_2A_2e_j\right|^2 -
\sum_{s=3}^k 
\frac{\lambda_s \left|P_sA_2e_j\right|^2}{1-\lambda_s}=\frac{1}{2}.
\end{equation}
Adding \eqref{squares_for_inverse} and
\eqref{squares_for A_1} we obtain
\begin{equation}\label{projection on complement}
\sum_{s=2}^k  \left| P_sA_2e_j \right| ^2=1
\end{equation}
Now, since \eqref{circle_order_one} means $P_1(A_2e_j)=0$,  
%\eqref{squares_for A_1} 
the displayed equation above yields
\begin{equation}\label{norm_one}
\left| A_2e_j\right| =1.
\end{equation}
Since $\Vert A_2\Vert = 1$, it follows from
\eqref{norm_one} that for every $1\leq j\leq n$ we have  
$e_j\in span \{ \xi_1,\dots, \xi_{2n}\}$. 
A similar analysis 
%applied to the decomposition near the spectral
%point $(-1,0)$   
shows that for each
$n + 1\leq j\leq 2n$ we have 
$e_j\in span \{ \xi_1,\dots, \xi_{2n}\}$. 
Thus,
\[ 
L =
span\{ \xi_1,\dots, \xi_{2n}\} =
span\{ e_1,\dots, e_{2n}\}
\] 
is a common invariant subspace of $A_1$ and $A_2$, which 
completes the proof of (1).  

Next, we %also
remark that the restrictions of $A_1$ and $A_2$ to $L$ are
involutions, and, since 
$
P_1A_2e_j=0
$
for  $j=1,\dots ,n$ and 
$ 
P_2A_2e_j=0 
$
for $j=n+1,\dots ,2n$,  
that
\begin{eqnarray}
A_2e_j\in span \{ e_{n+1},\dots,e_{2n} \}, \text{ for } j=1,\dots,n; \\ 
A_2e_j\in span \{e_1,\dots,e_n \}, \text{ for } j=n+1,\dots,2n.
\end{eqnarray}
Therefore, in the basis $e_1,\dots,e_{2n}$ the restrictions of 
$A_1$ and $A_2$ to $L$
%, which we denote by $C_1$ and $C_2$ respectively, 
look like
\begin{equation}
A_1|_L=
\left [ 
\begin{array}{ll}
 I_n & 0_n \\
 0_n & -I_n	
\end{array} 
\right ], 
\qquad 
A_2|_L=
\left [
\begin{array}{ll}
0_n & D_n \\
D_n^* & 0_n	
\end{array} 
\right ],
\end{equation}
where $I_n$ and $0_n$ are the identity and zero $n\times n$ matrices respectively, and $D_n$ is a unitary $n\times n $ matrix.
Write
\[
U=\left[
\begin{array}{cc}
     D_{n}^\ast & 0_{n} \\
         0_{n} & I_{n}
\end{array} 
\right ], 
\]
then $U$ is unitary and
\begin{equation}\label{pauli}
 U \left ( 
\left[
\begin{array}{cc}
          I_{n} & 0_{n} \\
           0_{n} & -I_{n}
\end{array}
\right], 
\left[
\begin{array}{cc}
   0_{n} & D_{n} \\
  D_{n}^\ast & 0_{n}
\end{array}
\right] \right )
U^\ast= \left ( \left[
\begin{array}{cc}
          I_{n} & 0_{n} \\
           0_{n} & -I_{n}
\end{array}
\right],
\left [ 
\begin{array}{cc}
          0_{n} & I_{n} \\
           I_{n} & 0_{n}
\end{array}
\right] \right ).  
\end{equation}
Now, the interchange of coordinate vectors 
$e_{2j} \leftrightarrow e_{2j-1+n}$ 
for $j=1,\dots,\lfloor\frac{n}{2}\rfloor$ finishes the proof of (2).

To prove (3) note that the relations $C_1^2=C_2^2=(C_1C_2)^4=1$  
%$$\left  ( C_1C_2 \right )^4=I$$
%is well-known and easily verified for matrices \eqref{pauli}.
are straightforward to verify. 
%We are done.
\end{proof}

We now prove a result that builds upon Lemma~\ref{dihedral}.

\begin{theorem}\label{reducing-unitary}
Let $A_1$ and $A_2$ be unitary self-adjoint 
%complex $N\times N$ matrices. 
linear operators on a finite-dimensional Hilbert space $V$. 
Then: 
\begin{enumerate}
\item
Every reduced component of  $\sigma_p(A_1,A_2)$ is either a 
line $\{x\pm y=\pm 1\}$ or an ``ellipse'' 
$\{x^2 +2xy\cos(2\pi\theta) + y^2 = 1\}$ for some $0<\theta<1/2$. 

\item 
If a line $\{x\pm y =\pm 1\}$ is a reduced component of multiplicity $r$ 
of the joint spectrum $\sigma_p(A_1,A_2)$ then $A_1$ and $A_2$ have 
a correponding common eigenspace of dimension $r$. 

\item 
If an ``ellipse"  \ 
%\begin{equation}%\label{ellipse3}
$
\{ x^2+2xy\cos(2\pi\theta) +y^2 =1 \} %\ 
$
\ with \ $0<\theta <1/2$ \ 
%\end{equation}
is a reduced component of the proper joint spectrum $\sigma_p(A_1,A_2)$ 
of multiplicity $r$, then $A_1$ and $A_2$ have a correponding 
%reducing 
common invariant subspace of dimension $2r$ 
%of dimension twice the multiplicity $r$ of the ellipse
that is a direct sum of $r$ two-dimensional %reducing 
common invariant subspaces.
\end{enumerate}
\end{theorem}

\begin{proof}
Since $A_1$ and $A_2$ are unitary involutions, the operator 
$R =A_1A_2+A_2A_1$ is self-adjoint of norm $\le 2$. Part (1) now 
is an immediate consequence of Lemma~\ref{dihedral}. 

We give the proof of (2) for the case when our line is  
$\{x+y=1\}$. The proof of the other cases is analogous. 

Choose a point $(\gamma, \beta)\in \{ x+y=1 \}$ with positive 
real coordinates. %sufficiently close to $(1,0)$. 
The spectrum of the operator $\gamma A_1+\beta A_2$ is the set 
$
\{ \lambda \in \R : \ (1/\lambda) (\gamma, \beta)\in \sigma_p (A_1,A_2)\}.
$ 
%It immediately follows from 
Since by part (1)  all nonlinear reduced components of 
$\sigma_p(A_1,A_2)$ are convex inside the positive real quadrant 
and intersect $\{x+y=1\}$ at $(1,0)$ and $(0,1)$,   
we have that 
$\max \{ |\lambda |: (1/\lambda) (\gamma, \beta) \in \sigma_p (A_1,A_2)\}$ corresponds to the point 
of the intersection of the lines 
$\{ x+y=1\}$ and $\{ \beta x-\gamma y=0\}$. Thus  
$
\max\{ |\lambda |: (1/\lambda)(\gamma, \beta) \in 
\sigma_p (A_1,A_2)\} = 1/(\gamma +\beta)=1.
$
Since the operator $\gamma A_1+\beta A_2$ is self-adjoint, this implies 
$\left\Vert\gamma A_1+\beta A_2\right\Vert = 1$. Also, this point of 
intersection is a regular point of $\sigma_p(A_1,A_2)$. 
By~\eqref{spectransform} 
the joint spectrum of $\gamma A_1+\beta A_2$ and $A_2$ 
also contains the line $\{x+y=1\}$ (with the same multiplicity $r$), and 
%the 
%a line passing through the point 
$(1,0)$ 
%and $(0,1)$ are regular on 
is a regular point of 
$\sigma_p(\gamma A_1+\beta A_2, A_2)$. Since 
$\left\Vert \gamma A_1+\beta A_2 \right\Vert =\left\Vert A_2\right\Vert =1$, 
by \cite[Proofs of Lemma 5 and Lemma 9]{CSZ} 
we have that $A_2$ and $\gamma A_1+\beta A_2$ have a common 
eigensubspace  corresponding to the eigenvalue $1$. %which  
By Proposition~\ref{R:multiplicity-eigenvalues} this eigenspace 
has dimension equal to the multiplicity of the line
$\{x+y=1\}$ in $\sigma_p(\gamma A_1+\beta A_2, A_2)$, hence is 
$r$-dimensional. Clearly it is an eigensubspace 
of eigenvalue $1$ for $A_1$ as well. %We are done.

Now we proceed with the proof of (3). In view of part (2), by 
restricting to orthogonal 
complements if necessary, we may assume without loss of generality that 
$\sigma_p(A_1,A_2)$ contains none of the lines $\{ x\pm y=1\}$ or 
$\{x \pm y= -1\}$. 
By part (1) and its proof, all reduced components of 
the proper joint spectrum are the nonsingular ``ellipses'' 
\begin{equation}\label{ellipse3}
\mathcal E_i=
\{
x^2 + 2xy\cos(2\pi\theta_i) + y^2 = 1
\}, 
\quad\text{with}\quad 
1/2 > \theta_1 > \dots > \theta_s > 0, 
\end{equation} 
their intersection 
points are $(\pm 1,0)$ and $(0,\pm 1)$, and   
%As in the proof of Lemma~\ref{dihedral} we observe that 
$\alpha_i=2\cos(2\pi\theta_i)$ is %for each $i$ 
an eigenvalue of $R =A_1A_2+A_2A_1$ for each $i$. Note that if
$M_i$ is the corresponding eigenspace of $R$, then for every 
$(x,y)\in \R^2$ 
on the ellipse $\mathcal E_i$ the space $M_i$ is invariant under
the action of $xA_1+yA_2$. Indeed, let $\xi \in M_i$. 
%Write
%$$
%u=(xA_1+yA_2)\xi.
%$$
Then
$$
(xA_1+yA_2)^2\xi=(x^2+y^2)\xi+xyR_1\xi=\xi.
$$
Thus, the restriction of $(xA_1+yA_2)^2$ to $M_i$ is the identity of $M_i$.
%By the spectral mapping theorem (cf. \cite{R}) 
Since the eigenvalues of $(xA_1+yA_2)^2$ are just the squares of the 
eigenvalues of $xA_1+yA_2$ with the same eigenvectors, every
eigenvector of $(xA_1+yA_2)^2$ with eigenvalue $1$ is a linear 
combination of eigenvectors of $(xA_1+yA_2)$
with eigenvalues $\pm 1 $. Conversely, 
the displayed above equality shows also that when $xy\ne 0$   
each eigenvector $\xi$ of $(xA_1+yA_2)$ with eigenvalue $\pm 1$,
is an eigenvector of $R$ with eigenvalue $\alpha_i$. Thus, 
when $xy\ne 0$ the eigenspace $M_i$ is the direct sum of 
the eigenspaces of $(xA_1+yA_2)$
with eigenvalues $\pm 1$, and, therefore, is invariant under $(xA_1+yA_2)$. 
In particular, since when $xy\ne 0$ the point $(x,y)$ belongs to only one 
reduced component (of multiplicity $n_i$) of $\sigma_p(A_1,A_2)$,  
is a nonsingular point on that component, 
and no lines through the origin are tangent to any of the reduced 
components of $\sigma_p(A_1,A_2)$, by 
Proposition~\ref{R:multiplicity-eigenvalues} the eigenspace of 
$xA_1+yA_2$ of eigenvalue $1$ has dimension $n_i$. Since our components 
are invariant under the transformation $(x,y)\mapsto (-x,-y)$, 
when $xy\ne 0$ the 
eigenspace of $xA_1+yA_2$ of eigenvalue $-1$ is also of dimension $n_i$;   
%whenever $xy\ne 0$,   
hence the dimension of $M_i$ is $2n_i$.

Next, 
the projective joint spectrum $\sigma(I,A_1,A_2)$ is the union of 
the homogenizations of the ``ellipses'' \eqref{ellipse3},  
with the same multiplicities. Furthermore, since there are 
no real  solutions to the equation 
$
x^2 + 2xy\cos(2\pi\theta_i) + y^2 = 0,  
$
there are no real points of the form
$[0 : x : y]$ in $\sigma(I,A_1,A_2)$, i.e. 
no real points on the ``infinite'' line $\{x_0=0\}$; and hence 
$0$ is not in the spectrum of $(xA_1+yA_2)$. 
Therefore 
%We remark that this case 
for every %pair 
$(x,y)\in \R^2$ with $(x,y)\neq (0,0)$ the operator $(xA_1+yA_2)$ 
is invertible. Also, 
if $(x,y)\in\mathbb R^2$ belongs to $\mathcal E_i$, then
\begin{equation}\label{ellipse4}
(xA_1+yA_2)^{-1}|_{M_i}=(xA_1+yA_2)|_{M_i}.
\end{equation}
Indeed, 
%Now, 
let %$(a,b)\in\R^2$ and 
$\xi \in M_i$. We have
$$
(xA_1+yA_2)^2\xi=(x^2+xy\alpha_i+y^2)\xi.
$$
Thus, $(xA_1+yA_2)^{-1}$ is defined by
\begin{equation}\label{ellipse5}
(xA_1+yA_2)^{-1}|_{M_i}=\frac{1}{x^2+xy\alpha_i +y^2}(xA_1+yA_2)|_{M_i},
\end{equation}
which, of course, implies \eqref{ellipse4}.

Next we note that
$$
x^2+2xy\cos(2\pi\theta_i) +y^2= 
\frac{1+\cos(2\pi\theta_i)}{2}(x+y)^2 + 
\frac{1-\cos(2\pi\theta_i)}{2}(x-y)^2,
$$
and, therefore, each ``ellipse'' \eqref{ellipse3} 
is centered at the origin,  
has axes along lines $\{x=y\}$ and $\{x=-y\}$,
the lengths of its semiaxes %are 
$\sqrt{\frac{1}{1+\cos(2\pi\theta_i)}}$ and 
$\sqrt{\frac{1}{1-\cos(2\pi\theta_i)}}$ respectively, and 
recall that we write $n_i$ for its multiplicity.

Set 
$$
\begin{array}{l}
B_1=\frac{1}{\sqrt{2\bigl(1+\cos(2\pi\theta_1)\bigr)}}(A_1+A_2)  \\
B_2=\frac{1}{\sqrt{2\bigl(1-\cos(2\pi\theta_1)\bigr)}}(A_1-A_2). 
\end{array}
$$
It is easy to check that the joint spectrum of $B_1$ and $B_2$ is
the union of the ``ellipses'' 
\begin{equation}\label{spectrum3}
\mathcal E_i' = 
\left \{ 
\frac{1+ \cos(2\pi\theta_i)}{1+ \cos(2\pi\theta_1)}x^2 +
\frac{1- \cos(2\pi\theta_i)}{1- \cos(2\pi\theta_1)}y^2 =1 
\right \}
\end{equation}
for $i=1,\dots,s$, each with multiplicity $n_i$. 
In particular, this implies that the unit circle $\{x^2+y^2=1\}$ is 
in $\sigma_p(B_1,B_2)$, has multiplicity $n_1$, and the spectrum of $B_2$ 
consists of numbers
$\{ \pm \sqrt{\frac{1- \cos(2\pi\theta_i)}{1- \cos(2\pi\theta_1)}} \} $. 
Hence, the norm of $B_2$ is equal to one. It also follows from 
\eqref{spectrum3} that
the points $(\pm 1,0)$ and $(0, \pm 1)$ are regular points of
$\sigma_p(B_1,B_2)$, all of them having multiplicity $n_1$.

Finally, if $(x,y)\in \mathbb R^2$ belongs to $\mathcal E'_i$, then
\begin{multline*}
xB_1+yB_2=
\left( 
\frac{1}{\sqrt{2\bigl(1+\cos(2\pi\theta_1)\bigr)}}x+
\frac{1}{\sqrt{2\bigl(1-\cos(2\pi\theta_1)\bigr)}}y 
\right )A_1 \\
+ 
\left( 
\frac{1}{\sqrt{2\bigl(1+\cos(2\pi\theta_1)\bigr)}}x- 
\frac{1}{\sqrt{2\bigl(1-\cos(2\pi\theta_1)\bigr)}}y 
\right )A_2=u_1A_1+u_2A_2.
\end{multline*}
It is easy to check that
$$ 
u_1^2+2u_1u_2\cos(2\pi\theta_i)+u_2^2= 1 %x^2+y^2=1, 
$$
which, of course, means that $(u_1,u_2)\in {\mathcal E}_i$. 
According to \eqref{ellipse4}
\begin{multline*}
(xB_1+yB_2)^{-1}|_{M_i}=(u_1A_1+u_2A_2)^{-1}|_{M_i} \\
=(u_1A_1+u_2A_2)|_{M_i}=(xB_1+yB_2)|_{M_i}. 
\end{multline*}
Let $\lambda_i=\sqrt{\frac{1 +\cos(2\pi\theta_1)}{1+ \cos(2\pi\theta_i)}}$. 
It follows that  $(\lambda_iB_1)^{-1}|_{M_i}=\lambda_iB_1|_{M_i}$, hence 
\begin{multline*}
(xB_1^{-1}+yB_2)|_{M_i}=
xB_1^{-1}|_{M_i}+yB_2|_{M_i} \\ 
=
x\lambda_i^2B_1|_{M_1}+yB_2|_{M_i}=
(x\lambda_i^2B_1+yB_2)|_{M_i}.  
\end{multline*} 
%and 
Therefore, except for finitely many points,  
when $(\lambda_i^2x, y)$ is on $\mathcal E_i'$ the 
space $M_i$ is the direct sum of eigenspaces of eigenvalues $\pm 1$ 
for the operator $(xB^{-1}_1+yB_2)$. Thus each ellipse 
$
\mathcal E_i''= 
\{
%\left \{ 
\frac{1+ \cos(2\pi\theta_1)}{1+ \cos(2\pi\theta_i)}
x^2 + \frac{1- \cos(2\pi\theta_i)}{1- \cos(2\pi\theta_1)}y^2 =1 
%\right \} 
\} 
$
is a reduced component of $\sigma_p(B_1^{-1},B_2)$, and by 
Proposition~\ref{R:multiplicity-eigenvalues} it has multiplicity $\ge n_i$. 
Since the sum of the multiplicities equals $(1/2)\dim V$, we must have that 
the multiplicity of $\mathcal E_i''$ equals $n_i$ for each $i$, and 
%that these are 
therefore these are all reduced components of $\sigma_p(B_1^{-1},B_2)$. 
In particular, the ``unit circle'' $\{x^2+y^2=1\}$ is a reduced component  
of $\sigma_p(B_1^{-1},B_2)$, of multiplicity $n_1$, and the points 
$(\pm 1,0)$ do not belong to any other of its components. 
It follows from Theorem~\ref{circle}
that $M_1$ is a direct sum of $n_1$ two-dimensional common 
%reducing 
invariant subspaces for $B_1$ and $B_2$ and the restriction 
of the pair $B_1,B_2$ 
to each of these subspaces is unitary equivalent to the pair
$$
\left[ 
\begin{array}{rr}
1 & 0 \\
0 & -1  
\end{array} 
\right ] , \ 
\left [ 
\begin{array}{ll}
0 & 1 \\
1 & 0 
\end{array} 
\right ] .
$$ 
Of course, each of these subspaces is also  
invariant under $A_1$ and $A_2$, and the restrictions on these 
subspaces of the pair $A_1,A_2$ is unitary equivalent to
$$
\left [ 
\begin{array}{rr}
\cos(\pi\theta_1) & \sin(\pi\theta_1) \\   
\sin(\pi\theta_1) & -\cos(\pi\theta_1) 
\end{array} 
\right ], \  
\left [ 
\begin{array}{rr}
\cos(\pi\theta_1) & -\sin(\pi\theta_1) \\
-\sin(\pi\theta_1) & -\cos(\pi\theta_1) 
\end{array} 
\right ].      
$$
Denote by ${\mathcal N}_1$ the orthocomplement to $M_1$. Then 
${\mathcal N}_1$ is invariant under $A_1$ and $A_2$ and 
the proper joint spectrum 
$
\sigma_p(A_1|_{{\mathcal N}_1},A_2|_{{\mathcal N}_1})
%=\sum_{j=2}^s n_j{\mathcal E}_j.
$
is the union of the ellipses $\mathcal E_j$ with $j=2,\dots,s$, each 
with multiplicity $n_j$.  
A similar argument applied to the subspace $M_2$ of ${\mathcal N}_1$ 
shows that $M_2$ is a direct sum of
$n_2$ two-dimensional common invariant subspaces for 
$A_1$ and $A_2$, and that their restrictions 
to those subspaces are
unitary equivalent to
\begin{equation}
\left [ 
\begin{array}{rr}
\cos(\pi\theta_2) & \sin(\pi\theta_2) \\
\sin(\pi\theta_2) & -\cos(\pi\theta_2) 
\end{array} 
\right ], \  
\left [ 
\begin{array}{rr}
\cos(\pi\theta_2) & -\sin(\pi\theta_2) \\
-\sin(\pi\theta_2) & -\cos(\pi\theta_2) 
\end{array} 
\right ].
\end{equation}
Since $\theta=\theta_j$ for some $j$, iterating the 
above procedure $j$ times yields the claimed result.
\end{proof}

As an immediate consequence of Theorem~\ref{reducing-unitary} 
and its proof we have:

\begin{theorem}\label{T:main-theorem-dihedral}
Let $G$ be a Coxeter group of type $I(n)$ for some $2\le n\le\infty$, 
let $\{g_1,g_2\}$ be a set of Coxeter generators, and let 
$\rho: G\rightarrow GL(V)$ be a finite dimensional unitary 
representation of $G$. 
\begin{enumerate} 
\item
If $\rho':G\rightarrow GL(V')$ is another finite dimensional 
uintary representation of $G$ such that 
$
\sigma_p\bigl(\rho(g_1),\rho(g_2)\bigr)=
\sigma_p\bigl(\rho'(g_1),\rho'(g_2)\bigr)
$ 
as subschemes of $\C^2$, then $\rho$ and $\rho'$ are equivalent. 

\item
If $\rho$ is irreducible then $\sigma_p\bigl(\rho(g_1),\rho(g_2)\bigr)$ 
is reduced and irreducible. \qed
\end{enumerate}
\end{theorem}

\section{Admissible transformations}\label{c-length} 
 
%\vspace{.2cm} 

For the rest of this paper $n\ge 2$, the group $G$ 
is a finite Coxeter group  of type $A_n$, $B_n$, or $D_{n+1}$, 
and $\{g_1,g_2,\dots \}$ is 
a set of Coxeter generators for $G$.  
%Let $G$ be a finite group and

For a word $w=g_{i_1}\dots g_{i_N}$, where the same 
letter $g_k$ might occur more 
than once, we write $|w|=N$. Let $a_k$ be the number of times $g_k$ 
occurs in $w$. The \emph{signature} of $w$ is the sequence 
\[
\sign(w) = (a_1, a_2, \dots ), 
\]
and the \emph{content} of $w$ is the sequence 
\[
\cont(w) = ( |w|, a_1, \dotsm, a_{n-1} )\in \mathbb N^n. 
\]
Since $|w|= a_1 + a_2 + \dots $, in types $A_n$ and $B_n$ the content 
and the signature of $w$ carry the same information, while in type $D_{n+1}$ 
it is clearly possible to have words with the same content but 
different signatures. Let $\le_{lex}$ be the lexicographic ordering on 
$\mathbb N^n$, that is, $c=(c_0,\dots,c_{n-1})<_{lex}(d_0,\dots,d_{n-1})=d$ 
precisely when  $c\ne d$ and $c_m<d_m$ for $m=\min\{i\mid c_i\ne d_i\}$. 
We always consider our words partially ordered $<$ by 
\[
w_1< w_2  \quad\text{if and only if}\quad
\cont(w_1)<_{lex} \cont(w_2). 
\]
For example, $g_2g_1g_2 < g_1g_2g_1$ whenever $n\ge 2$.

%\vspace{.2cm} 
\begin{definition}
(a)   
We introduce \emph{admissible transformations} on words. They are: 
\begin{enumerate}  
%\vspace{.2cm} 
% 
%\noindent 
\item %0) 
Cancelling transformations. That is, for each  $i$ the transformation 
\[
w'g_ig_iw'' \mapsto w'w'' 
\]
is admissible. %\\ 

\item %1) 
Commuting transformations. This means that if 
$g_i$ and $g_j$ commute, then the transformation 
\[
w'g_ig_jw'' \mapsto w'g_jg_iw''
\] 
is admissible. %\\

\item %2) 
Circular transformations. Those are: 
\[ 
g_{i_1}\dots g_{i_k}g_{i_{k+1}}\dots g_{i_N} \mapsto 
g_{i_{k+1}}\dots g_{i_N}g_{i_1}\dots g_{i_k}.
\] %\\ 

\item %3) 
Replacement transformations. These 
replace a certain subword consisting of a successive string of letters 
by another representation of this subword. Specifically,   
\[
\begin{aligned} 
w'g_ig_{i+1}g_iw'' &\mapsto w'g_{i+1}g_ig_{i+1}w'' &   
&\text{for $1 \le i\le n-2$};     \\
w'g_{n-1}g_ng_{n-1}w'' &\mapsto w'g_ng_{n-1}g_nw'' &   
&\text{in $A_n$ and $D_{n+1}$}; \\
w'g_{n-1}g_{n+1}g_{n-1}w'' &\mapsto w'g_{n+1}g_{n-1}g_{n+1}w'' &  
&\text{in $D_{n+1}$}; %\\ 
%w'g_{n-1}g_ng_{n-1}g_nw'' &\mapsto w'g_ng_{n-1}w'' &   
%&\text{in $B_n$};
\end{aligned} 
\]
and in addition the ``tent commuting'' replacements  
\[
\begin{aligned} 
w't_kg_jw'' &\mapsto w'g_jt_kw''  \\ 
&\qquad\qquad\text{for $k<j\le n-1$ in $B_n$ and $D_{n+1}$}; \\
w't_kg_nw'' &\mapsto w'g_{n+1}t_kw'' \ \text{ and }  \\ %\quad 
w't_kg_{n+1}w'' &\mapsto w'g_nt_kw''  \\ 
&\qquad\qquad\text{for $k\le n-1$ in $D_{n+1}$}; \\
w't_kg_nw'' &\mapsto w'g_nt_kw'' \\  
&\qquad\qquad\text{for $k\le n-1$ in $B_n$} 
\end{aligned}
\] 
are admissible transformations. 
Here the 
$k$th ``tent word'' $t_k$ is defined for $1\le k\le n-1$ as 
\[
t_k= 
\begin{cases}
g_k\dots g_{n-1}g_ng_{n-1} \dots g_k  &\text{in } B_n; \\ 
g_k\dots g_{n-1}g_ng_{n+1}g_{n-1}\dots g_k &\text{in } D_{n+1}.  
\end{cases} 
\]
%\vspace{.2cm} 
\end{enumerate} 

(b) 
We write $w_1\leadsto w_2$ if there is a sequence of admissible 
transformations that maps $w_1$ into $w_2$. Taking the equivalence 
closure of this transitive relation on words, we obtain the notion  
of {\it $c$-equivalence}. We write $w_1\sim w_2$ if the words   
$w_1$ and $w_2$ are $c$-equivalent.
\end{definition}

%\vspace{.2cm} 
\begin{remark}\label{R:admissible-facts}
(a) 
Note that in $B_n$ and $D_{n+1}$ we have the following 
``tent commuting'' equalities of elements. First,  
\[
\begin{aligned}
t_kg_j &= u g_{j-1}g_jt_{j+1}g_jg_{j-1}(v g_j)   \\   
      &= u g_{j-1}g_jt_{j+1}(g_jg_{j-1}g_j) v   \\
      &= u g_{j-1}g_j(t_{j+1}g_{j-1})g_jg_{j-1}v \\   
      &= u (g_{j-1}g_jg_{j-1})t_{j+1}g_jg_{j-1}v \\  
      &= (u g_j)g_{j-1}g_jt_{j+1}g_jg_{j-1}v    \\ 
      &= g_j(ug_{j-1}g_jt_{j+1}g_jg_{j-1}v)     \\ 
      &= g_j t_k
\end{aligned}   
\]  
whenever $k<j\le n-1$, as $u$ and $v$ only involve letters $g_i$ with 
$i\le j-2$ hence they commute with $g_j$, while $t_{j+1}$ only involves 
$g_i$s with $i\ge j+1$, hence it commutes with $g_{j-1}$. For trivial 
reasons, the same equality also holds when $j+1 < k\le n-1$. 
%for trivial reasons. 
In a similar manner, when 
$k\le n-1$ we have in $D_{n+1}$ the equalities 
\[
\begin{aligned}
t_kg_n &= u g_{n-1}(g_ng_{n+1})g_{n-1}(v g_n)   \\   
      &= u g_{n-1}g_{n+1}(g_ng_{n-1}g_n) v   \\
      &= u (g_{n-1}g_{n+1}g_{n-1})g_ng_{n-1}v \\   
      &= (u g_{n+1})g_{n-1}(g_{n+1}g_n)g_{n-1}v \\  
%      &= (u g_j)g_{j-1}g_jt_{j+1}g_jg_{j-1}v    \\ 
      &= g_{n+1}(ug_{n-1}g_ng_{n+1}g_{n-1}v)     \\ 
      &= g_{n+1} t_k
\end{aligned}   
\]
and $t_kg_{n+1}=g_nt_k$. Finally, in $B_n$ we have the equalities 
\[
\begin{aligned}
t_kg_n &= u g_{n-1}g_ng_{n-1}(v g_n)   \\   
      &= u (g_{n-1}g_ng_{n-1}g_n) v   \\
      &= (u g_n)g_{n-1}g_ng_{n-1}v \\   
%      &= u (g_{j-1}g_jg_{j-1})t_{j+1}g_jg_{j-1}v \\  
%      &= (u g_j)g_{j-1}g_jt_{j+1}g_jg_{j-1}v    \\ 
      &= g_n(ug_{n-1}g_ng_{n-1}v)     \\ 
      &= g_n t_k  
\end{aligned}   
\]
for all $k\le n-1$. 

(b) 
Combining the above yields in $B_n$ and $D_{n+1}$ also the tent commuting 
equalities 
\[
t_kt_j = t_j t_k 
\]
whenever $1\le k, j\le n-1$. 

(c)
It is immediate from (a) and the defining relations of $G$ 
that if $w_1\leadsto w_2$ and no circular transformations 
were used, then $w_1$ and $w_2$ represent the same element of 
the group $G$. 

(d)
It is also clear from (a) 
that if $w_1$ is $c$-equivalent to $w_2$ then these words 
represent elements of $G$ that 
belong to the same conjugacy class. The converse to this is also true, 
but since we will not need it here, we leave the easy though somewhat 
cumbersome proof as an exercise for 
the interested reader. 
\end{remark}

We conclude this section with a simple but crucially important 
observation.  

\begin{lemma}\label{L:admissible-reduces} 
If \ %for two words $w_1$ and $w_2$ we have 
$w_1\leadsto w_2$ \ 
then \ $\cont(w_1)\ge_{lex}\cont(w_2)$. %in the partial ordering.  
\end{lemma} 

\begin{proof} 
Indeed every listed admissible transformation either preserves 
the content of the word, or lowers it. 
\end{proof}

\section{Echelon forms}
\label{S:echelon-forms}

Now we will show that words in our Coxeter group 
$G$ have conjugates in a special form that we 
call \textbf{echelon form}.

\begin{definition}\label{D:echelon-form} 
We say that a word $w$ is in \emph{echelon form} if it is of the 
form 
\[
w=\delta_1\delta_2\dots\delta_n\delta_{n+1} 
\]
where for each $i$ the word $\delta_i$ satisfies 
\begin{equation}\label{E:delta-conditions} 
\delta_i \in 
\begin{cases} 
\quad \{1, g_i\}      &\text{if $i\le n$ and $G=A_n$}; \\ 
\quad \{1 \}          &\text{if $i=n+1$ and $G=A_n$ or $G=B_n$}; \\
\quad \{1, g_n\}      &\text{if $i=n$ and $G=B_n$}; \\   
\quad \{1, g_i\}      &\text{if $i\ge n$ and $G=D_{n+1}$}; \\ 
\quad \{1, g_i, t_i\} &\text{if $i\le n-1$ and $G=B_n$ or $G=D_{n+1}$}.  
\end{cases} 
\end{equation}
\end{definition}

\begin{remark}\label{R:content-determines-echelon}  
Note that when $w$ is in echelon form, it can be recovered 
from $\cont(w)$ in types $A_n$ and $B_n$, while in type $D_{n+1}$ 
there are at most two possibilities for $w$ given its  
content. More precisely, in all types,  
given $\cont(w)=(|w|, a_1,\dots, a_{n-1})$, 
the tuple $(a_1,\dots, a_i)$ determines the tuple  
$(\delta_1,\dots,\delta_i)$ whenever $1\le i\le n-1$. In types 
$A_n$ and $B_n$ the tuple $(\delta_1,\dots,\delta_{n-1})$ together 
with $|w|$ determines $\delta_n$ and hence $w$. In type $D_{n+1}$ 
the tuple $(\delta_1,\dots,\delta_{n-1})$ together with $|w|$ leaves 
at most two possibilities for $(\delta_n, \delta_{n+1})$, 
depending on the parity of $a_n+a_{n+1}$. If this parity is even, 
there is no choice. If the parity is odd, then there is a choice 
between either $(g_n, 1)$ or $(1, g_{n+1})$ for 
$(\delta_n,\delta_{n+1})$.       
\end{remark} 

We can extend  the last remark even further. 

\begin{proposition}\label{P:echelon-conjugates} 
Let $w_1$ and $w_2$ be two words in $D_{n+1}$, both in echelon form, and  
suppose that $\cont(w_1)=\cont(w_2)=(W,a_1,\dots, a_{n-1})$. 

If $W-a_1-\dots - a_{n-1}\ge 3$ then $w_1$ and $w_2$ 
%, hence they  
belong to the same conjugacy class as elements of $G$.  
\end{proposition} 

\begin{proof} 
By Remark~\ref{R:content-determines-echelon} we may assume 
that $a_n+a_{n+1}=W-a_1-\dots - a_{n-1}$ is odd and that 
\[
w_1=\delta_1\dots\delta_{n-1}g_n 
\quad\text{and}\quad 
w_2=\delta_1\dots\delta_{n-1}g_{n+1}, 
\] 
where each $\delta_i$ satisfies \eqref{E:delta-conditions}. 
Since $a_n+a_{n+1}\ge 3$, we must have  
$\delta_i=t_i$ for at least one index $i$. Let $k$ be 
the biggest such index, and let $j$ be the second biggest 
(if such exists). Then $\delta_i\in\{1,g_i\}$ for each $j< i<n$
with $i\ne k$, hence the tent commuting equalities from 
Remark~\ref{R:admissible-facts}  yield 
\begin{align*} 
w_1 &=  
\delta_1\dots\delta_{j-1}t_j(\delta_{j+1}\dots \delta_{k-1})t_k
                          (\delta_{k+1}\dots \delta_{n-1})g_n \\
    &= 
\delta_1\dots\delta_{j-1}t_j(\delta_{k+1}\dots \delta_{n-1})
                          (\delta_{j+1}\dots \delta_{k-1})t_kg_n \\
    &= 
\delta_1\dots\delta_{j-1}t_j(\delta_{k+1}\dots \delta_{n-1})
                          (\delta_{j+1}\dots \delta_{k-1})g_{n+1}t_k \\ 
    &= 
\delta_1\dots\delta_{j-1}t_j(\delta_{k+1}\dots \delta_{n-1})g_{n+1} 
                          (\delta_{j+1} \dots \delta_{k-1})t_k. \\  
\intertext{Now applying a circular transformation we obtain}
    &\sim  
(\delta_{j+1} \dots \delta_{k-1})t_k \delta_1\dots\delta_{j-1}t_j
(\delta_{k+1}\dots \delta_{n-1})g_{n+1}\\  
%t_j\delta_1\dots\delta_{j-1}\delta_{j+1}\dots\delta_{n-1}g_{n+1} \\ 
\intertext{and using again the tent commuting equalities we get} 
    &=
\delta_1\dots\delta_{j-1}t_j(\delta_{j+1}\dots \delta_{k-1})t_k
                          (\delta_{k+1}\dots \delta_{n-1})g_{n+1} 
    = w_2  
\end{align*} 
as desired. 
\end{proof}

The following key theorem is the main result of this section.

\begin{theorem}[Ordering Theorem]\label{T:ordering-theorem}  
Let $w$ be a word. There exists a word $\tilde w$ in echelon 
form such that $w\leadsto \tilde w$. 
\end{theorem}

The proof relies on the following lemma. 

\begin{lemma}\label{L:reduction-step} 
Let $i$ be the smallest index of a letter in \ $w$. 
Then \ $w\leadsto w'\delta_i w''$ \ where the words $w'$ and $w''$ 
can contain only letters with index $> i$, and $\delta_i$ satisfies 
\eqref{E:delta-conditions}. Furthermore, the sequence of admissible 
transformations can be chosen so that no circular transformations 
are used. 
\end{lemma}

\begin{proof} 
The case when $i\ge n$ is trivial, so we may assume that $i\le n-1$,  
hence without loss of generality also that $i=1$. If $a_1=1$ then 
no transformations are needed, so we may assume that $a_1\ge 2$. Thus  
we have $w= ug_1vg_1s$ where $u$ and $v$ do not involve $g_1$.  

Consider now the case $n=2$. We prove this case by induction on $a_1$.   
Suppose first that $a_1=2$. Then $s$ does not involve $g_1$.  
Since $v$ can involve only $g_2$ and/or $g_3$, 
a sequence of commuting and/or cancelling transformations 
yields $w\leadsto ug_1g_2g_1s$ or, in case of $D_3$, 
either $w\leadsto ug_1g_2g_3g_1s$ or $w\leadsto ug_1g_3g_1s$. 
When $G=A_2$ or $G=D_3$  a replacement transformation gives 
$ug_1g_2g_1s\leadsto ug_2g_1g_2s$ which is in the desired form. 
When $G=B_2$ we already have $ug_1g_2g_1s=ut_1 s$ is in the 
desired form. When $G=D_3$ then $ug_1g_2g_3g_1s=ut_1s$ is also 
in the desired form, and a replacement transformation yields 
$ug_1g_3g_1s\leadsto ug_3g_1g_3s$ which is also in the desired form.    
Suppose next that $a_1\ge 3$. Then $s=s'g_1s''$ where $s''$ does 
not involve $g_1$. By induction we have $ug_1vg_1s'\leadsto u'\delta_1v'$ 
where $u'$ and $v'$ do not involve $g_1$ and 
no circular transformations were used, hence the same 
sequence of transformations yields $w\leadsto u'\delta_1 v'g_1s''$. 
Since $v'$ can involve only $g_2$ and/or $g_3$, 
a sequence of cancelling and/or commuting transformations yields 
$w\leadsto u'\delta_1g_2g_1s''$ or, in case of $D_3$, also either 
$w\leadsto u'\delta_1g_2g_3g_1s''$ or $w\leadsto u'\delta_1g_3g_1s''$. 
In case $A_2$ we have $\delta_1=g_1$ hence  a replacement transformation 
brings us to $w\leadsto u'g_2 g_1 g_2 s''$ which is in the desired form. 
In case $B_2$ and $\delta_1=g_1$ then we are already have 
$u'g_1g_2g_1s''=u't_1s''$ which is in the desired form. 
In case $B_2$ and $\delta_1=t_1=g_1g_2g_1$ 
then a tent commuting transformation followed by cancelling yields  
$
u't_1g_2g_1s'' \leadsto 
u'g_2t_1g_1s'' = u'g_2g_1g_2g_1g_1s'' \leadsto 
u'g_2g_1g_2s''
$ 
which is again in the desired form. 
In the case $D_3$ and $\delta_1=g_1$ we handle the possibilities   
$u'g_1g_2g_1s''$ and $u'g_1g_3g_1s''$ just as in the case $A_2$, and 
the possibility $u'g_1g_2g_3g_1s''=u't_1s''$ is already in the desired form. 
%with  $\delta_1=t_1=g_1g_2g_3g_1$.  
Finally, we look at the case $D_3$ and 
$\delta_1=t_1=g_1g_2g_3g_1$. Then using tent commuting transformations yields 
$u't_1v'g_1s'' \leadsto u'v''t_1g_1s''$ where $v''$ is obtained by exchanging  
the letters $g_2$ and $g_3$ in $v'$. Then, a cancelling transformation 
produces $u'v''t_1g_1s''=u'v''g_1g_2g_3g_1g_1s''\leadsto u'v''g_1g_2g_3s''$ 
which is again in the desired form and completes the proof of the case $n=2$.  
  
Next, suppose $n\ge 3$. 
By induction on $n$ we have that $v\leadsto v'\delta_2 v''$ where 
$v'$ and $v''$ can only involve letters with index $3$ or higher, and 
no circular transformations are used.   
%(when $n=3$ this follows from the case $i\ge n$). 
Therefore the 
same sequence of admissible transformations yields 
$w \leadsto ug_1v'\delta_2 v''g_1t$. Since $g_1$ commutes with all 
the letters appearing in $v'$ and $v''$, a sequence of commuting 
transformations yields $w\leadsto uv'g_1\delta_2 g_1v''t$. 
If $\delta_2=1$ then a cancelling transformation yields 
$w\leadsto uv'v''t$; if $\delta_2=g_2$ then a replacement transformation 
gives $w\leadsto uv'g_2g_1g_2v''t$; and if $\delta_2=t_2$ then we 
have $uv'g_1t_2g_1v''t = uv't_1v''t$. 
If $a_1=2$ then all three cases above are in the desired form. 
If $a_1\ge 3$ then the first two cases have a smaller $a_1$ and in 
them we are done by induction on $a_1$. In the last case, induction 
on $a_1$ gives $v''t\leadsto z'\delta_1z''$ where $z'$ and $z''$ do 
not involve $g_1$, and no circular transformations are used. Therefore 
the same sequence of transformations yields 
$w\leadsto uv't_1z'\delta_1 z''$, hence a sequence of 
tent commuting replacement 
transformations produces $w\leadsto u''t_1\delta_1z''$. Now, if 
$\delta_1=g_1$ then $t_1\delta_1=g_1t_2g_1g_1$ hence a cancelling 
transformation yields $w\leadsto u''g_1t_2z''$ which is in the 
desired form. Finally, if $\delta_1=t_1$ we get 
$t_1\delta_1=g_1t_2g_1g_1t_2g_1$ hence a cancelling transformation 
yields $w\leadsto u''g_1t_2t_2g_1z''$ which has smaller $a_1$ hence 
we are done by induction on $a_1$ in this case as well.  
This completes the proof of the lemma.   
\end{proof}

\begin{proof}[Proof of Theorem~\ref{T:ordering-theorem}] 
Let $i_1$ be the smallest index of a letter in $w$. 
By Lemma~\ref{L:reduction-step} we get $w\leadsto w'\delta_{i_1}w''$ 
hence a circular trasnformation produces 
$w\leadsto \delta_{i_1}w''w' = \delta_1\dots\delta_{i_1}w_{i_1}$ where 
$w_{i_1}$ only involves letters with index $> i_1$. Suppose for some $k\le n-1$ 
we have already obtained $w\leadsto \delta_1\dots\delta_kw_k$ where 
$w_k$ can only involve letters with index $>k$. If $k=n-1$ then the  
only possible letters involved in $w_{n-1}$ are $g_n$ and $g_{n+1}$, so 
we are done after a sequence of commuting and/or cancelling transformations.    
If $k< n-1$ then by Lemma~\ref{L:reduction-step} we have 
$w_k\leadsto w_k'\delta_{k+1}w_k''$ where no circular transformations 
are used, and $w_k'$ and $w_k''$ can involve only letters with index $> k+1$. 
Thus the same sequence of transformations yields 
$w\leadsto \delta_1\dots\delta_kw_k'\delta_{k+1}w_k''$. 
Therefore a sequence of commuting and/or tent commuting transformations 
produces $w\leadsto z\delta_1\dots\delta_{k+1}w_k''$ where $z$ can involve 
only letters with index $>k+1$. Now a circular transformation gives 
$w\leadsto \delta_1\dots\delta_{k+1}w_{k+1}$ where $w_{k+1}=w_k''z$ can 
involve only letters of index $> k+1$. Iterating this process we 
arrive at the desired conclusion. 
\end{proof}

\section{The proof of Theorem~\ref{main-theorem} }\label{main theorem} 
 
\vspace{.2cm} 

We are now ready to present the proof of our second main result.

\begin{proof}[Proof of Theorem~\ref{main-theorem}] 
By Theorem~\ref{unitary} we may assume that both $\rho_1$ and $\rho_2$ 
are unitary. Also, the equality $D(T,\rho_1)=D(T,\rho_2)$ implies 
$\dim V_1=\deg D(T,\rho_1)=\deg D(T,\rho_2)=\dim V_2$, and 
%In order to prove this it suffices to show that if 
\begin{equation}\label{equal spectra}
\sigma_p\bigl(\rho_1(g_1),\dots,\rho_1(g_n)\bigr)=
\sigma_p\bigl(\rho_2(g_1),\dots,\rho_2(g_n)\bigr).
\end{equation} 
%where $g_1,\dots,g_n$ are Coxeter generators of the corresponding group, 
In particular, the case when $G$ is finite of type $I(n)$ is an 
immediate consequence from Theorem~\ref{T:main-theorem-dihedral}. 

In the remaining cases for $G$ it suffices by 
Theorem~\ref{character determines representation} 
to show that the equality of subschemes %\eqref{equal spectra} 
$D(T,\rho_1)=D(T,\rho_2)$ implies  
that the characters %of the representations $\rho_1$ and $\rho_2$, \ 
$\chi_{\rho_1}$ and $\chi_{\rho_2}$ are the same.

%\vspace{.2cm}
Let $N=\dim V_1=\dim V_2$. 
For each $i$ let $A_i=\rho_1(g_i)$ and let $B_i=\rho_2(g_i)$. 
For every $(x_1,\dots, x_n)\in\C^n$ the characteristic polynomials 
%the linear terms in the defining polynomials 
for the operators 
$x_1A_1+\dots+ x_nA_n$ and $x_1B_1+\dots+ x_nB_n$ are 
equal. Therefore   
the spectra of 
$x_1A_1+\dots+x_nA_n$ and $x_1B_1+\dots+x_nB_n$ are the same counting 
multiplicities. This implies that the traces of these operators are 
the same, in other words
$$
x_1\trace(A_1)+\dots+x_n\trace(A_n)	= 
x_1\trace(B_1)+\dots+x_n\trace(B_n).
$$
Since this is true for all $x_1,\dots,x_n$ we obtain
\begin{equation}\label{traces1}
\trace(A_j)	= \trace(B_j), \ j=1,\dots,n.
\end{equation}
Similarly, since for each $m$ the eigenvalues of $(x_1A_1+\dots+x_nA_n)^m$ 
are just the $m$th powers of the eigenvalues of $x_1A_1+\dots+x_nA_n$, we get 
that  
%Now, the spectral mapping theorem (see \cite{L}, p.) that 
for every $m\in {\mathbb N}$ the spectra of $(x_1A_1+\dots+x_nA_n)^m$ and $(x_1B_1+\dots+x_nB_n)^m$ are the same counting multiplicities.  %
%and again, 
Since this is true for arbitrary $(x_1,\dots,x_n)\in \C^n$, we obtain 
%a similar argument shows 
that for every $m$ and $\alpha=(a_1,\dots,a_n)\in {\mathbb N}^n$ with 
$a_1+\dots+a_n=m$ one has 
\begin{equation}\label{traces2}
\sum_{\sign(u)=\alpha} \trace\bigl(\rho_1(u)\bigr) =  
%A_{j_1}A_{j_2}...A_{j_m})= 
\sum_{\sign(u)=\alpha} \trace\bigl(\rho_2(u)\bigr) 
%B_{j_1}B_{j_2}...AB_{j_m}),
\end{equation}

To complete the proof it is enough 
to show by induction on the partial ordering of words 
that for each word $w$ in $G$ we have 
\[
\trace\bigl(\rho_1(w)\bigr) = \trace\bigl(\rho_2(w)\bigr). 
\]
When $w=1$ this statement is trivial since both sides equal $N$. 
Suppose $w\ne 1$ and we have proved our statement for all words 
$< w$. Let $m=|w|$ and $\sign(w)=\alpha=(a_1,a_2,\dots )$. 
For each word $u$ with 
$\sign(u)=\alpha$ let $\tilde u$ be a word in echelon form 
such that $u\leadsto \tilde u$, as per the Ordering Theorem. 
Thus, by Lemma~\ref{L:admissible-reduces} we have 
$\cont(w)=\cont(u)\ge_{lex}\cont(\tilde u)$. 
Let 
\[
M=\bigl|\{u \mid \sign(u)=\alpha\text{ and } 
                 \cont(\tilde u)=\cont(w) \}\bigr|.
\] 
Since $u$ and $\tilde u$ belong to the same conjugacy class in $G$, 
we can rewrite equation \eqref{traces2} as follows: 
%\[
\begin{multline}\label{E:multline} 
\sum_{\substack{
      \sign(u)=\alpha \\ 
      \tilde u< w}}
\trace\bigl(\rho_1(\tilde u)\bigr) 
\quad + 
\sum_{\substack{
      \sign(u)=\alpha \\ 
      \cont(\tilde u)=\cont(w)}}
\trace\bigl(\rho_1(\tilde u)\bigr)                  \\
=
\sum_{\substack{
      \sign(u)=\alpha \\ 
      \tilde u < w}} 
\trace\bigl(\rho_2(\tilde u)\bigr) 
\quad + 
\sum_{\substack{
      \sign(u)=\alpha \\ 
      \cont(\tilde u) = \cont(w)}} 
\trace\bigl(\rho_2(\tilde u)\bigr). 
\end{multline} 
%\]  
If $\tilde w< w$ then by induction hypothesis we get 
\[
\trace\bigl(\rho_1(w)\bigr) = 
\trace\bigl(\rho_1(\tilde w)\bigr)=
\trace\bigl(\rho_2(\tilde w)\bigr)= 
\trace\bigl(\rho_2(w)\bigr),
\] 
so we may assume that $\cont(w)=\cont(\tilde w)$.  
Also, note that, in type $D_{n+1}$, if 
$a_n=0$ (resp. $a_{n+1}=0$) and $\sign(u)=\alpha$, then 
by the nature of the possible admissible transformations, 
$\tilde u$ does not involve $g_n$ (resp. $g_{n+1}$).  
Therefore, in view of Remark~\ref{R:content-determines-echelon} 
and Proposition~\ref{P:echelon-conjugates}, 
in all possible cases for $u$ with  
$\sign(u)=\alpha$ and $\cont(\tilde u)=\cont(w)$ we must have 
$\tilde u$ in the same conjugacy class as $w$. Therefore
$\trace\bigl(\rho_i(\tilde u)\bigr)=\trace\bigl(\rho_i(w)\bigr)$ 
hence 
\[
\sum_{\substack{
      \sign(u)=\alpha \\ 
      \cont(\tilde u)=\cont(w)}}
\trace\bigl(\rho_i(\tilde u)\bigr)   
=M\trace\bigl(\rho_i(w)\bigr)
\]  
for $i=1,2$. Since by the induction hypothesis 
$
\trace\bigl(\rho_1(\tilde u)\bigr)=
\trace\bigl(\rho_2(\tilde u)\bigr)
$ 
whenever $\tilde u < w$, equation \eqref{E:multline} reduces to 
\[
M\trace\bigl(\rho_1(w)\bigr)=M\trace\bigl(\rho_2(w)\bigr)
\]
and the desired conclusion is immediate. 
\end{proof}

We also note the following result which was obtained during the  
course of establishing equality \eqref{traces2} in the proof above.

\begin{theorem}\label{traces3}
Let $G$ be any finitely generated group 
with generators $g_1,\dots,g_n$, 
and let $\rho_1$ and $\rho_2$ be two finite dimensional 
representations of $G$. 
If the equality \eqref{equal spectra} holds, then for every 
$(k_1,\dots,k_n)\in {\mathbb N}^n$ we have 
\begin{equation}\label{character1}
\sum_{\sign(w)=(k_1,\dots,k_n) }\chi_{\rho_1}(w) %\nonumber \\ 
\quad = \quad 
\sum_{\sign(w)=(k_1,\dots,k_n)} \chi_{\rho_2}(w),	 
\end{equation}
where $w$ denotes a word in the alphabet given by the set 
$\{g_1,\dots, g_n\}$.
\end{theorem}

\section{Example: $\widetilde{C}_2$}\label{affine c}

Recall that $\widetilde{C}_2$ is the Coxeter group with three generators $b_1,b_2,b_3$ that satisfy the relations: 
\begin{equation}\label{relations} 
b_j^2=1, \ j=1,2,3; \quad (b_1b_2)^4=(b_2b_3)^4=1; \quad b_1b_3=b_3b_1. 
\end{equation} 
It is well known that this group is \emph{affine}, that is, 
it contains an abelian 
normal subgroup such that the quotient group is finite. 
In this particular case it is easy to see that this normal 
abelian subgroup is the following: let 
$$
r_1=b_1b_2b_3b_2, \quad  r_2=b_2b_1b_2b_3.
$$ 
Let us check that the group generated by $r_1$ and $r_2$ is an abelian 
normal subgroup of $\widetilde{C}_2$. First we observe that the relations
\eqref{relations} imply 
\begin{equation}\label{br}
\begin{split} 
b_1r_1&=r_1^{-1}b_1 \\ 
b_2r_1&=r_2b_2,    \\ 
b_3r_1&=r_1b_3,    \\ 
b_1r_2&=r_2b_1      \\ 
b_2r_2&=r_1b_2,     \\ 
b_3r_2&=r_2^{-1}b_3  \\ 
b_3r_2^{-1}&=r_2b_3  
\end{split}
\end{equation} 
 
We now can use these last relations to establish that %commuting 
$r_1$ and $r_2$ commute: 
$$
r_1r_2=r_1b_2b_1b_2b_3=b_2r_2b_1b_2b_3=b_2b_1r_2b_2b_3
=b_2b_1b_2r_1b_3=b_2b_1b_2b_3r_1=r_2r_1.
$$ 
Now, relations \eqref{br} imply that 
\begin{align*} 
b_1r_1b_1&=r_1^{-1} \\ 
b_2r_1b_2&=r_2 \\ 
b_3r_1b_3&=r_1 \\ 
b_1r_2b_1&=r_2 \\ 
b_2r_2b_2&=r_1 \\ 
b_3r_2b_3&=r_2^{-1} 
\end{align*} 
These last relations show that the subgroup ${\mathcal N}$ of 
$\widetilde{C}_2$ 
generated by $r_1$ and $r_2$ is a normal subgroup.  
We further remark that $b_1=b_2b_3b_2r_1^{-1}$, and, therefore, the cosets of $b_2$ and $b_3$ generate $\widetilde{C}_2/{\mathcal N}$. 
We will now show: 

\begin{theorem}
If for two finite dimensional unitary representations 
$\rho_1$ and $\rho_2$ of the affine Coxeter group 
$\widetilde{C}_2$ we have 
\begin{multline*}
\sigma_p
\bigl(
\rho_1(b_2),\rho_1(b_3),
\rho_1(r_1),\rho_1(r_2),\rho_1(r_1^{-1}),\rho_1(r_2^{-1})
\bigr)
\\ 
= \sigma_p
\bigl(
\rho_2(b_2),\rho_2(b_3),
\rho_2(r_1),\rho_2(r_2),\rho_2(r_1^{-1}),\rho_2(r_2^{-1})
\bigr), 
\end{multline*}
%the joint spectra of the images of $b_2,b_3,r_1,r_2,r_1^{-1}$, and $r_2^{-1}$ 
%under $\rho_1$ and $\rho_2$ are the same, 
then for these representations 
%are equivalent. Since the character determines a representation of an affine 
%Coxeter group up to equivalence, we are again to show that 
$\chi_{\rho_1}=\chi_{\rho_2}$. 
\end{theorem}

\begin{proof} 
First we note that the subgroup generated by $b_2$ and $b_3$ is the dihedral group $I(4)=B_2$. The only words in echelon form here are  
$1$, $b_2$, $b_3$, $b_2b_3$, $b_2b_3b_2$, and $b_2b_3b_2b_3$. 
Note that $b_3$ and $b_2b_3b_2$ belong to the same conjugacy class. 
Since every element of $G$ can be written as $wr_1^{m_1}r_2^{m_2}$ for 
some word $w$ in $b_2$ and $b_3$, 
formulae \eqref{br} and 
Theorem~\ref{T:ordering-theorem} imply that   every element in 
$\widetilde{C}_2$ has in its conjugacy class a word in the form 
$$
b_2^{k_1}b_3^{l_1}b_2^{k_2}b_3^{k_2}r_1^{m_1}r_2^{m_2}
$$ 
where $k_1,k_2,l_1\in\{0,1\}, \ k_1\geq k_2, \ l_1\geq k_2$, 
and $m_1,m_2 \in {\mathbb Z}$. 
%and $l_2=0$, if $k_2=0$. 
In other words, each word in $\widetilde{C}_2$ has in its conjugacy class 
one of the words 
\begin{align*}
               w_0(m_1,m_2) &=r_1^{m_1}r_2^{m_2} , \\ 
%\text{or\quad} 
               w_1(m_1,m_2) &=b_2r_1^{m_1}r_2^{m_2}, \\ 
%\text{or\quad} 
               w_2(m_1,m_2) &=b_3r_1^{m_1}r_2^{m_2}, \\ 
%\text{or\quad} 
               w_3(m_1,m_2) &= b_2b_3r_1^{m_1}r_2^{m_2}, 
\text{\quad\quad or\quad}\\  
               w_4(m_1,m_2) &=b_2b_3b_2b_3r_1^{m_1}r_2^{m_2}. 
\end{align*} 
Thus, we are to show that if 
\begin{multline*} 
\sigma\bigl(-I,\rho_1(b_2),\rho_1(b_3),\rho_1(r_1),\rho_1(r_2),
       \rho_1(r_1^{-1}),\rho_1(r_2^{-1})\bigr) \\ 
= \sigma\bigl(-I,\rho_2(b_2),\rho_2(b_3),\rho_2(r_1),\rho_2(r_2),
         \rho_2(r_1^{-1}),\rho_2(r_2^{-1})\bigr), 
\end{multline*} 
then for every $m_1,m_2\in {\mathbb Z}$ and every $j=0,1,2,3,4$ 
\begin{equation}\label{main-equality}
\chi_{\rho_1}\bigl(w_j(m_1,m_2)\bigr)=\chi_{\rho_2}\bigl(w_j(m_1,m_2)\bigr). 
\end{equation}

A similar argument to the one in Theorem~\ref{traces3} shows that 
for every $\eta\in\mathbb N^6$ with 
$\eta=(n_1,n_2,n_3,n_4,n_5,n_6)$ 
\begin{equation}\label{sum_characters}	 
\sum_{\sign(w)=\eta} \chi_{\rho_1}(w)=\sum_{\sign(w)=\eta} \chi_{\rho_2}(w). 
\end{equation} 
where the sums are taken over all words comprised of 
$n_1$ \ $b_2$-s, $n_2$ \ $b_3$-s, $n_3$ \ $r_1$-s, 
$n_4$ \ $r_2$-s, $n_5$ \ $r_1^{-1}$-s, and $n_6$ \ $r_2^{-1}$-s. 
 
In particular, using circular transformation if necessary, 
we see that in the cases $n_1=1, n_2=n_5=n_6=0$ all terms 
in each of the sums 
\eqref{sum_characters} are the same. Therefore, 
for all $m_1,m_2\geq 0$ 
\begin{equation}\label{characters1}
\chi_{\rho_1}\bigl(w_1(m_1,m_2)\bigr)=\chi_{\rho_2}\bigl(w_1(m_1,m_2)\bigr).
\end{equation} 
A similar argument shows that the same equality holds for 
$m_1\geq 0, m_2\leq 0$, \ $m_1\leq 0, m_2 \geq 0$, \ $m_1,m_2 \leq 0$, 
and, hence, 
\eqref{characters1} holds for all $m_1,m_2\in {\mathbb Z}$. 
 
Applying the same reasoning to the case $n_1=0, n_2=1$ we obtain 
\begin{equation}\label{characters2} 
\chi_{\rho_1}\bigl(w_j(m_1,m_2)\bigr)=\chi_{\rho_2}\bigl(w_j(m_1,m_2)\bigr). 
\end{equation} 
holds for $j=2$ as well. 
 
Let us show that \eqref{characters2} holds for $j=3$. We will prove it by induction in $s=|m_1|+|m_2|$. If $s=0$, \ $w_3(0,0)$ belongs to the subgroup generated by $b_2$ and $b_3$. Since the joint spectrum of $\rho_j(b_2)$, \ $\rho_j(b_3)$ and the identity is the intersection of 
$
\sigma\bigl(-I,\rho_1(b_2),\rho_1(b_3),\rho_1(r_1),\rho_1(r_2),
       \rho_1(r_1^{-1}),\rho_1(r_2^{-1})\bigr) 
$ 
with the plane $\{ x_3=x_4=x_5=x_6=0 \}$, we see that 
$$
\sigma\bigl(-I,\rho_1(b_2),\rho_1(b_3)\bigr)=
\sigma\bigl(-I,\rho_2(b_2),\rho_2(b_3)\bigr),
$$ 
so by our previous results 
$$
\chi_{\rho_1}(b_2b_3)=\chi_{\rho_2}(b_2b_3).
$$ 
 
Now, let us fix $m_1,m_2\geq 0$ with $m_1+m_2>0$, and consider $n_1=n_2=1, \ n_3=m_1, \ n_4=m_2, \ n_5=n_6=0$. In each side of 
\eqref{sum_characters} there are characters of words 
in one of two forms:  
\begin{equation}\label{form}
%\begin{matrix} 
r_1^{q_1}r_2^{t_1}b_2r_1^{q_2}r_2^{t_2}b_3r_1^{q_3}r_2^{t_3}, %\\ %\nonumber \\ 
\text{\qquad or\qquad} %\\ %\label{form} \\ 
r_1^{q_1}r_2^{t_1}b_3r_1^{q_2}r_2^{t_2}b_2r_1^{q_3}r_2^{t_3}, %\nonumber 
%\end{matrix}
\end{equation} 
where $q_j,t_j\geq 0, \ q_1+q_2+q_3=m_1, \ t_1+t_2+t_3=m_2$ 
(we observe that the order of $r_1$ and $r_2$ does not 
matter since they commute). 
 
Using circular transformations shows that each word that appears in the second form in \eqref{form} is in the same conjugacy class as the first one, of course with different $q_j$ and $t_j$ which still sum up to $m_1$ and $m_2$ respectively. Thus, it suffices to show the equality of characters for words in the first form in \eqref{form}. 
 
We use circular transformation to get 
\begin{equation}\label{6}
\begin{split} 
r_1^{q_1}r_2^{t_1}b_2r_1^{q_2}r_2^{t_2}b_3r_1^{q_3}r_2^{t_3} &\sim b_2r_1^{q-2}r_2^{t_2}b_3r_1^{q_3}r_2^{t_3}r_1^{q_1}r_2^{t_1}         \\ 
 &= b_2r_1^{q_2}r_2^{t_2}b_3r_1^{q_1+q_3}r_2^{t_1+t_3}.  
\end{split}
\end{equation} 
By \eqref{br} $r_2b_3=b_3r_2^{-1}$, so  we can get from \eqref{6} 
\begin{equation*} 
r_1^{q_1}r_2^{t_1}b_2r_1^{q_2}r_2^{t_2}b_3r_1^{q_3}r_2^{t_3} \sim b_2r_1^{q_2}b_3r_1^{q_1+q_3}r_2^{t_1-t_2+t_3}. 
\end{equation*} 
Since by \eqref{br} $r_1$ and $b_3$ commute, we finally obtain 
\begin{equation}\label{8} 
r_1^{q_1}r_2^{t_1}b_2r_1^{q_2}r_2^{t_2}b_3r_1^{q_3}r_2^{t_3} \sim b_2b_3r_1^{q_1+q_2+q_3}r_2^{t_1-t_2+t_3}=b_2b_3r_1^{m_1}r_2^{t_1-t_2+t_3}. 
\end{equation} 
\begin{comment}
Similarly, 
$$r_1^{q_1}r_2^{t_1}b_3r_1^{q_2}r_2^{t_2}b_2r_1^{q_3}r_2^{t_3}\sim r_1^{q_2}r_2^{t_2}b_2r_1^{q_3+q_1}r_2^{t_3+t_1}b_3$$, 
so by \eqref{8} we obtain 
\begin{equation}\label{9} 
r_1^{q_1}r_2^{t_1}b_3r_1^{q_2}r_2^{t_2}b_2r_1^{q_3}r_2^{t_3}\sim b_2b_3r_1^{m_1}r_2^{-t_1+t_2-t_3}. 
\end{equation} 
\end{comment} 
If $|t_1-t_2+t_3|<m_2$, by induction 
\begin{equation}\label{lower} 	\chi_{\rho_1}(r_1^{q_1}r_2^{t_1}b_2r_1^{q_2}r_2^{t_2}b_3r_1^{q_3}r_2^{t_3})=
\chi_{\rho_2}(r_1^{q_1}r_2^{t_1}b_2r_1^{q_2}r_2^{t_2}b_3r_1^{q_3}r_2^{t_3}). 
\end{equation} 
\begin{comment} 
which implies 
\begin{eqnarray} 
	\sum_{q_1+q_2+q_3=m_1, \  |t_1-t_2+t_3|=m_2, \ t_1+t_2+t_3=m_2} \chi_{\rho_1}(r_1^{q_1}r_2^{t_1}b_2r_1^{q_2}r_2^{t_2}b_3r_1^{q_3}r_2^{t_3})= \nonumber \\ 
	 \sum_{q_1+q_2+q_3=m_1, \  |t_1-t_2+t_3|=m_2, \ t_1+t_2+t_3=m_2}\chi_{\rho_2}(r_1^{q_1}r_2^{t_1}b_2r_1^{q_2}r_2^{t_2}b_3r_1^{q_3}r_2^{t_3}), \label{11} 
\end{eqnarray} 
where in the last relation all indices $q_1,q_2,q_3,t_1,t_2$ and $t_3$ are non-negative. 
\end{comment} 
Of course, for non-negative $t_1,t_2,t_3$ 
\begin{equation}\label{irreducible}
t_1+t_2+t_3=|t_1-t_2+t_3| \ \Longrightarrow \ t_2=0, \ \mbox{or} \ t_1=t_3=0. 
\end{equation} 
 
\begin{comment} 
Thus, we can rewrite \eqref{11} as 
\begin{eqnarray} 
	\sum_{q_1+q_2+q_3=m_1,\ 0\leq t\leq m_2} \chi_{\rho_1}(r_1^{q_1}r_2^{t}b_2r_1^{q_2}b_3r_1^{q_3}r_2^{m_2-t}) \nonumber \\ +\sum_{q_1+q_2+q_3=m_1}\chi_{\rho_1}(r_1^{q_1}b_2r_1^{q_2}r_2^{m_2}b_3r_1^{q_3} \nonumber \\ 
	= \sum_{q_1+q_2+q_3=m_1,\ 0\leq t\leq m_2} \chi_{\rho_2}(r_1^{q_1}r_2^{t}b_2r_1^{q_2}b_3r_1^{q_3}r_2^{m_2-t}) \label{12}\\ 
	+ \sum_{q_1+q_2+q_3=m_1}\chi_{\rho_2}(r_1^{q_1}b_2r_1^{q_2}r_2^{m_2}b_3r_1^{q_3}. \nonumber 
\end{eqnarray} 
 
Now, (\ref{8}) implies that all terms in the first sum in both sides of (\ref{12}) are the same, and so  are the terms of the second sums. More precisely, each term of the first sum on the left is equal to $\chi_{\rho_1}(b_2b_3r_1^{m_1}r_2^{m_2}$, and each term of the second sum in the left-hand side of (\ref{12}) is equal to $\chi_{\rho_1}(b_2b_3$.

% Uncomment the following two lines if you want to have a bibliography 
%\bibliographystyle{alpha} 
%\bibliography{document} 
\end{comment} 
 
We will now show that every term in \eqref{sum_characters} that are not conjugates of words in the form \eqref{form} with a smaller $m_1+m_2$ 
is in the same conjugacy class as $b_2b_3r_1^{m_1+m_2}$. Of course, 
this together with \eqref{sum_characters} and \eqref{lower} will imply 
that \eqref{main-equality} holds for $j=3$. 
 
By \eqref{irreducible} for such a word either $t-2=0$, that is the 
word is in the form 
\begin{equation}\label{first-form}
\begin{split}
r_1^{q_1}r_2^tb_2r_1^{q_2}b_3r_1^{q_3}r_2^{m_2-t} &=
r_1^{q_1}r_2^tb_2b_3r_1^{m_1-q_1}r_2^{m_2-t}  \\ 
&\sim b_2b_3 r_1^{m_1}r_2^{m_2},  
\end{split}
\end{equation} 
(here we used the fact that $b_3$ and $r_1$ commute and a circular transformation), or $t_1=t_3=0$, so the word is 
\begin{equation}\label{second-form}
\begin{split} 
r_1^{q_1}b_2r_1^{q_2}r_2^{m_2}b_3r_1^{q_3}
&=r_1^{q_1}b_2r_2^{m_2}r_1^{q_1}b_3r_1^{q_3} \\ 
&=r_1^{q_1}b_2r_2^{m_2}b_3r_1^{q_2+q_3}  \\ 
&\sim b_2r_2^{m_2}b_3r_1^{m_1},  
\end{split}
\end{equation} 
here we again we used the commuting of $b_3$ and $r_1$ and a circular transformation that moved $r_1^{q_1}$ to the end of the word. 
 
Now, by \eqref{br} $b_2r_2=r_1b_2$ and $r_2b_2=b_2r_1$, therefore, in \eqref{first-form} we have 
$$
b_2b_3r_1^{m_1}r2^{m_2}\sim r_2^{m_2}b_2b_3r_1^{m_1}=b_2r_1^{m_2}b_3r_1^{m_1}=b_2b_3r_1^{m_1+m_2}.
$$ 
For \eqref{second-form} we have 
$$
b_2r_2^{m_2}b_3r_1^{m_1}=r_1^{m_2}b_2b_3r_1^{m_1}\sim b_2b_3r_1^{m_1+m_2}.
$$ 
This finishes the proof of \eqref{main-equality} 
for $j=3$ and$m_1,m_2\geq 0$. The proofs in the cases 
when one of, or both $m_1$ and $m_2$ are negative are 
practically identical with the only difference being that for 
$m_1<0$  \ $r_1$ is replaced with $r_1^{-1}$, and $r_2$ is 
replaced with $r_2^{-1}$ when $m_2<0$. 
 
We now prove that \eqref{main-equality} holds for $j=4$. 
Again, for every pair of representations with the same joint spectrum of $b_2,b_3,r_1,r_2,r_1^{-1},r_2^{-1}$ the restrictions of these representations to the subgroup generated by $b_2$ and $b_3$ are unitary equivalent 
%by Theorem ... , 
so 
$$
\chi_{\rho_1}(b_2b_3b_2b_3)=\chi_{\rho_2}(b_2b_3b_2b_3).
$$ 
Next we observe that relations \eqref{br} imply 
$$
b_2r_1b_2=r_2, \ b_2r_2b_2=r_1, \ b_2r_1^{-1}b_2=
r_2^{-1}, \ b_2r_2^{-1}b_2=r_1^{-1},
$$ 
and, therefore, 
\begin{equation}\label{15} 
b_2r_1^kb_2=r_2^k, \ b_2r_2^kb_2=r_1^k, \ k\in {\mathbb Z}. 
\end{equation} 
Relations \eqref{15} imply 
\begin{equation}\label{interchange} 
b_2b_3b_2b_3r_1^kr_2^l\sim b_2b_3b_2b_3r_1^lr_2^k, \ k,l\in {\mathbb Z}. 
\end{equation} 
Indeed, 
\begin{multline*} 
b_2b_3b_2b_3r_1^kr_2^l=b_2b_3b_2b_3b_2r_2^kb_2b_2r_1^lb_2
=b_2b_3b_2b_3b_2r_2^kr_1^lb_2 \\ 
=b_3b_2b_3r_2^kr_1^lb_2\sim b_2b_3b_2b_3r_1^lr_2^k. 
\end{multline*} 
We also remark that 
\begin{equation}\label{reduction} 
b_2b_3b_2b_3r_1^kr_2^l\sim b_2b_3b_2b_3r_1^{k-2}r_2^l
%\sim b_2b_3b_2b_3r_1^kr_2^{l-2}	\nonumber \\ 
%\sim b_2b_3b_2b_3r_1^{k-2}r_2^{l-2},  
\end{equation} 
and, similarly 
\begin{eqnarray}\label{reduction1} 
b_2b_3b_2b_3r_1^kr_2^l\sim b_2b_3b_2b_3r_1^kr_2^{l-2}. 
%                    \sim b_2b_3b_2b_3r_1^kr_2^{l-2}	\nonumber \\ 
%\sim b_2b_3b_2b_3r_1^{k-2}r_2^{l-2}.  
\end{eqnarray} 
Verification of \eqref{reduction} goes as follows: 
\begin{align*} 
b_2b_3b_2b_3r_1^kr_2^l &= 
b_2b_3b_2b_3r_1^{k-1}r_2^lr_1 \\ 
&\sim r_1 b_2	b_3b_2b_3r_1^{k-1}r_2^l \\ 
&=b_2r_2b_3b_2b_3r_1^{k-1}r_2^l  \\
&=b_2b_3r_2^{-1}b_2b_3r_1^{k-1}r_2^l \\ 
&=b_2b_3b_2r_1^{-1}b_3r_1^{k-1}r_2^l \\ 
&= b_2	b_3b_2b_3r_1^{k-2}r_2^l. 
\end{align*} 
Relation \eqref{reduction1} is an immediate 
consequence of \eqref{reduction} and \eqref{interchange}. 
%verified in a similar way. 
Now, \eqref{reduction} and \eqref{reduction1} show that 
if $|k|+|l|\geq 2$, and $(k,l)\neq (\pm 1,\pm 1)$, then 
\begin{equation}\label{induction} 
	b_2b_3b_2b_3r_1^kr_2^l\sim b_2b_3b_2b_3r_1^{k_1}r_2^{l_1} 
\mbox{ \ \ with} \ \  \ |k_1|+|l_1|<|k|+|l|.	 
\end{equation} 
Suppose that $m_1,m_2\geq 0$. Each term in \eqref{sum_characters} with $n_1=2,n_2=2,n_3=m_1,n_4=m_2,n_5=n_6=0$ is in the form 
\begin{equation}\label{form1}
r_1^{q_1}r_2^{t_1}b_{j_1}r_1^{q_2}r_2^{t_2}b_{j_2}
r_1^{q_3}r_2^{t3}b_{j_3}r_1^{q_4}r_2^{t_4}b_{j_4}r_1^{q_5}r_2^{t_5},
\end{equation} 
where $q_k,t_k\geq 0, \ \sum q_k=m_1, \ \sum t_k=m_2, \ j_k= \{ 2,3\}$ 
and there are two $b_2$ among $b_{j_k}$ and two $b_3$. Observe that, 
if for some $k=1,2,3$ \ $j_k=j_{k+1}$, then formulae \eqref{br} show 
that the word is a conjugate of either $w_3(l_1,l_2)$, or of 
$w_0(l_1,l_2)$ for some $l_1,l_2$. Since for such words we have 
already proven the equality of the characters, the sums of terms in \eqref{sum_characters} with $j_k\neq j_{k+1}$ on the left and on 
the right are the same. Each of such therms is in the form \eqref{form1} 
with $b_{j_1}=b_2, b_{j_2}=b_3, b_{j_3}=b_2, b_{j_4}=b_3$, or with 
$b_{j_1}=b_3, b_{j_2}=b_2,b_{j_3}=b_3,b_{j_4}=b_2$. Using circular 
transformations we easily see that the latter one is a conjugate 
of the former 
(possibly with different $q$-s and $t$-s). As it was mentioned above, 
relations \eqref{br} shows that each such term is a conjugate of 
$$b_2b_3b_2b_3r_1^kr_2^l$$ 
for some $k$ and $l$. 
 
A similar argument shows that the same conjugacy relation holds for arbitrary $m_1,m_2\in {\mathbb Z}$. 
 
Now we use an induction argument in $|m_1|+|m_2|$. 
We have already proven that \eqref{main-equality} holds for $w_4(0,0)$. 
Relation \eqref{induction} shows that it suffices to prove it for 
$w_4(\pm 1,0), w_4(0,\pm 1)$ and $w_4(\pm 1,\pm 1)$. 
 
Obviously, \eqref{interchange}--\eqref{reduction1} imply 
\begin{eqnarray} 
w_4(1,0)\sim w_4(0,1)\sim w_4(-1,0)\sim w_4(0,-1), \label{1,0} \\ 
w_4(1,1)\sim w_4(-1,1)\sim w_4(1,-1)\sim w_4(-1,-1). \label{1,1} 
\end{eqnarray} 
All terms in relation \eqref{sum_characters} with $n_1=2,n_2=2,n_3=1,n_4=n_5=n_5=0$ that are not conjugates 
of $w_3(k,l)$ or $w_0(k,l)$ are 
\begin{align*} 
r_1b_2b_3b_2b_3, \quad r_1b_3b_2b_3b_2, \\ 
b_2r_1b_3b_2b_3, \quad b_3r_1b_2b_3b_2, \\ 
b_2b_3r_1b_2b_3, \quad b_3b_2r_1b_3b_2, \\ 
b_2b_3b_2r_1b_3, \quad b_3b_2b_3r_1b_2, \\ 
b_2b_3b_2b_3r_1, \quad b_3b_2b_3b_2r_1.	 
\end{align*} 
It is easy to see using \eqref{br} that each of them is a conjugate of one of $w_4(\pm 1,0)$ or $w_4(0,\pm 1)$. By \eqref{1,0} all of them are in the same conjugate class and, therefore, \eqref{main-equality} implies 
$$
\chi_{\rho_1}\bigl(w_4(\pm 1,0)\bigr)=
\chi_{\rho_2}\bigl(w_4(\pm 1,0)\bigr), \ 
\chi_{\rho_1}\bigl(w_4(0,\pm 1)\bigr)=
\chi_{\rho_2}\bigl(w_4(0,\pm 1)\bigr).
$$ 
Similarly, we can show that each term in \eqref{sum_characters} 
corresponding to $n_1=2,n_2=2,n_3=n_4=1,n_5=n_6=0$ is either a conjugate 
of $w_4(0,0)$, or of $w_4(\pm 2,0)\sim w_4(0,0)$, or of 
$w_4(0, \pm 20\sim w_4(0,0)$, or $w_4(\pm 1,\pm 1)$. 
For the first three the equality of characters has been established. 
For the rest this equality follows from \eqref{main-equality} and 
\eqref{1,1}. %We are done. 
\end{proof}

\end{document}